\documentclass{amsart}

\usepackage[english]{babel} 
\usepackage[utf8]{inputenc} 
\usepackage[T1]{fontenc}
\usepackage{lmodern} 
\usepackage{hyperref}

\usepackage{xcolor}

\hypersetup{
    colorlinks = true,
    linkbordercolor = {red},
    linkcolor ={blue},
    anchorcolor = {pink},
    citecolor =  {red},
    filecolor = {blue},
    menucolor = {blue},
    runcolor =  {blue},
    urlcolor = {blue},
}
\usepackage{thmtools} 
\usepackage{amssymb} 
\usepackage{amsmath} 
\usepackage{mathrsfs} 
\usepackage{todonotes}
\usepackage{tikz-cd}
\usepackage{hyperref}
\usepackage{mathtools}
\usepackage[capitalize]{cleveref} 

\newtheorem{theorem}{Theorem}[section]
\newtheorem{lemma}[theorem]{Lemma}
\newtheorem{proposition}[theorem]{Proposition}
\newtheorem{corollary}[theorem]{Corollary}

\newtheorem{claim}[theorem]{Claim}
\newtheorem{construction}[theorem]{Construction}

\theoremstyle{definition}
\newtheorem{definition}[theorem]{Definition}
\newtheorem{notation}[theorem]{Notation}
\newtheorem{example}[theorem]{Example}

\newtheorem{remark}[theorem]{Remark}

\numberwithin{equation}{section}

\newenvironment{claimproof}[1][Proof of Claim]{\begin{proof}[#1]}{\end{proof}}

\newcommand{\cat}[1]{\mathsf{#1}}

\newcommand{\Set}{\cat{Set}}
\newcommand{\CH}{\cat{CH}}
\newcommand{\CompOrd}{\cat{CompOrd}}
\newcommand{\CoAlg}{\cat{CoAlg}}
\newcommand{\Alg}{\cat{Alg}}
\newcommand{\op}{\mathrm{op}}

\newcommand{\EM}{\mathrm{EM}}

\newcommand{\df}{\coloneqq}

\newcommand{\V}{\mathbb{V}}
\newcommand{\Vl}{\mathbb{V}^{\downarrow}}
\newcommand{\Vu}{\mathbb{V}^{\uparrow}}
\newcommand{\Vc}{\mathbb{V}^{\mathrm{c}}}

\let\d\relax
\DeclareMathOperator{\d}{{\downarrow}}
\let\u\relax
\DeclareMathOperator{\u}{{\uparrow}}
\DeclareMathOperator{\ud}{{\updownarrow}}

\DeclareMathOperator{\SP}{SP}

\newcommand{\dG}{\partial}

\newcommand{\StComp}{\cat{StComp}}

\newcommand{\MA}{\cat{MA}}
\newcommand{\BA}{\cat{BA}}

\DeclareMathOperator{\ima}{im}
\DeclareMathOperator{\colim}{colim}

\title{Duality for coalgebras for Vietoris and monadicity}
\author[M. Abbadini]{Marco Abbadini}
\address{School of Computer Science,
University of Birmingham,
B15 2TT Birmingham, United Kingdom}
\email{marco.abbadini.uni@gmail.com}
\author[I. Di Liberti]{Ivan Di Liberti}
\address{Department of Mathematics,
Stockholm University,
Stockholm, Sweden
}
\email{diliberti.math@gmail.com}

\begin{document}
\maketitle

\begin{abstract}
    We prove that the opposite of the category of coalgebras for the Vietoris endofunctor on the category of compact Hausdorff spaces is monadic over $\Set$. We deliver an analogous result for the upper, lower and convex Vietoris endofunctors acting on the category of stably compact spaces. We provide axiomatizations of the associated (infinitary) varieties.
    This can be seen as a version of J\'onsson-Tarski duality for modal algebras beyond the 0-dimensional setting.
\end{abstract}

   {
   \hypersetup{linkcolor=black}
   \tableofcontents
   }

\section*{Introduction}

After the original contribution by Stone \cite{Stone1936}, duality theory bloomed in several directions and with different motivations. A research line investigates dualities between algebra and geometry, offering representation theorems for geometric objects into (possibly infinitary) varieties. Among those, Duskin duality \cite{Duskin1969} shows that the category of compact Hausdorff spaces is dually equivalent to a(n infinitary) variety and has opened the door for several variations of Stone-like duality. Duskin duality has provided a more algebraic interpretation for Tietze extension theorem, Stone-Weierstrass theorem \cite{reggio2021beth} and several other results of this kind, leading to a duality theory based on the closed interval $[0,1]$ in place of the more usual Sierpis\'nki space \cite{HofmannNora2023}.
This research direction has been very active in the past years, including some contributions of the first author who has recently proved that the opposite of the category of Nachbin's compact ordered spaces and continuous order-preserving maps is a variety too \cite{Abbadini2019,AbbadiniReggio2020}.

Another research direction was initiated by J\'onsson and Tarski in 1951 \cite{JonssonTarski1951,JonssonTarski1952}, building on the logical interpretation of Stone duality and offering a topological representation of modal logic. This approach was later perfected by Esakia \cite{Esakia1974} and is nowadays a milestone of topological methods in modal logic \cite{blackburn2001modal}. In its present state, this framework is packaged in the analysis of several variations of the \textit{Vietoris} functor and its coalgebras. A synthetic and modern way to state Esakia's version of J\'onsson-Tarski duality is to say that the category of coalgebras for the Vietoris endofunctor over Boolean spaces is dually equivalent to that of modal algebras \cite{VenemaVosmaer2014}.

In the past years there has been a growing interest in amalgamating these two research lines \cite{bezhanishvili2015modal,bezhanishvili2022modal}. Indeed, the Vietoris functor admits very natural extensions to several categories of compact spaces, and thus the study of \textit{Stone duality above dimension} zero \cite{marra2017stone} has attracted a lot of attention. In \cite{HofmannNevesEtAl2019,HofmannNevesEtAl2018} the authors show quasivariety results for some of these categories. 

This paper offers a more grounded counterpart of \cite{kurz2002modal}:  we study the category of coalgebras for (several variations of) the Vietoris functor and deliver several positive variety results. More specifically we study the categories $\CH$ of compact Hausdorff spaces, $\CompOrd$ of Nachbin's compact ordered spaces and $\StComp$ of stably compact spaces. We obtain the following results.

\begin{itemize}
    \item[\ref{t:monadicity-vietoris}] The opposite $\CoAlg(\V)^\op$ of the category  of coalgebras for the Vietoris functor $\V$ on compact Hausdorff spaces is monadic over $\Set$.
    \item[\ref{t:compord-monadicity}] The opposite $\CoAlg(\Vc)^\op$ of the category of coalgebras for the convex Vietoris functor $\Vc$ on compact ordered spaces is monadic over $\Set$.
    \item[\ref{t:stcomp-monadicity}] The opposites $\CoAlg(\Vu)^\op$ and $\CoAlg(\Vl)^\op$  of the categories of coalgebras for the upper Vietoris $\Vu$ and lower Vietoris functor $\Vl$ on stably compact spaces is monadic over $\Set$.
\end{itemize}

Besides these very abstract results, we combine methods coming from categorical logic and general topology to provide a complete axiomatization of these categories of algebras as follows.
\begin{itemize}
    \item 
    The algebraic theory of $\CoAlg(\V)^\op$ can be obtained by adding to the algebraic theory of $\CH^\op$ the unary operator $\Box$ (or, equivalently, $\Diamond$) and appropriate axioms.
    \item 
    The algebraic theory of $\CoAlg(\Vc)^\op$ can be obtained by adding to the algebraic theory of $\CompOrd^\op$ the unary operators $\Box$ and $\Diamond$ and appropriate axioms.
    \item 
    The algebraic theory of $\CoAlg(\Vu)^\op$ can be obtained by adding to the algebraic theory of $\StComp^\op$ the unary operator $\Box$ and appropriate axioms.
    \item
    The algebraic theory of $\CoAlg(\Vl)^\op$ can be obtained by adding to the algebraic theory of $\StComp^\op$ the unary operator $\Diamond$ and appropriate axioms.
\end{itemize}

\subsubsection*{Structure of the paper}

In the first section, we briefly recall the categorical technology that sits at the core of our proof strategy. In the second section, we study the Vietoris functor on compact Hausdorff spaces and show the monadicity of its opposite category of coalgebras. The third section is devoted to the convex Vietoris functor, acting on compact ordered spaces, while the fourth section is devoted to the upper and lower Vietoris functors. Finally, the last section discusses the axiomatization of these opposite categories of coalgebras.

\section{Endofunctors, algebras and monadicity} \label{s:algebras}

In this brief section, we recall the most relevant results of the theory of \textit{variators}. We will use this technology to ground our main results in the next sections. Most of the content is expository, or folklore, and is organised in the most convenient way for our purposes.

\subsection{Algebras for an endofunctor and their monadicity}

Recall that, given an endofunctor $F \colon \cat{C} \to \cat{C}$, we can define its category of algebras $\Alg(F)$, whose objects are pairs $(X, f \colon FX \to X)$ and a morphism from $(X,f)$ to $(Y,g)$ is a map $h \colon X \to Y$ in $\cat{C}$ that respect the structure of algebras, i.e the diagram below is commutative. 

\[\begin{tikzcd}[ampersand replacement=\&]
	FX \& FY \\
	X \& Y
	\arrow["h"', from=2-1, to=2-2]
	\arrow["f"', from=1-1, to=2-1]
	\arrow["g", from=1-2, to=2-2]
	\arrow["Fh", from=1-1, to=1-2]
\end{tikzcd}\]

Variants of this theory have been considered in the literature, for example when the endofunctor is pointed $\eta \colon 1 \to F$. In that case, the notion of algebra has to change accordingly, but the general theory remains quite similar. We shall discuss in the last section the logical implications of these technical differences.  We refer to \cite{Adamek1974,trnkova1975free,Barr1970,barr2019coequalizers} for a general introduction to the topic.

\begin{remark}
    It comes with no surprise that the theory contained in this section admits a straightforward dualization to the case of coalgebras for an endofunctor and comonadicity for a comonad.
\end{remark}

\begin{remark}
    Given an endofunctor $F \colon \cat{C} \to \cat{C}$, we have a forgetful functor $U \colon \Alg(F) \to \cat{C}$ that maps an object $f \colon FX \to X$ to $X$, and a morphism $h \colon X \to Y$ (from $f \colon FX \to X$ to $g \colon FY \to Y$) to $h$ itself.
\end{remark}
	
One of the most compelling problems of this theory since its very start is to decide whether the forgetful functor in the remark above is monadic. Notably, these categories are monadic \textit{whenever they can}, in the sense clarified by the theorem below.

\begin{theorem}[{\cite[Corollary~5.10]{Barr1970}}]\label{t:varietor-monadic}
    The forgetful functor $\Alg(F) \to \cat{C}$ is monadic if and only if it is a right adjoint.
\end{theorem}

We recall from \cite{AdamekTrnkova1990} that a functor $F \colon \cat{C} \to \cat{C}$ is called a \emph{varietor} if the forgetful functor $\Alg(F) \to \cat{C}$ is right adjoint (and hence, by \cref{t:varietor-monadic}, monadic).
In this case, the corresponding monad---denoted by $F^*$--- is called the \emph{algebraically-free monad on $F$}.
Usually, the algebraically-free monad is described via the following construction (see \cite[p.~592]{Adamek1974}, \cite[IV.3.2]{AdamekTrnkova1990} or \cite[Construction~3.10]{AdamekPorst2003}).

\begin{construction}[Ad\'amek's free-algebra construction]
    Let $\cat{C}$ be a cocomplete category.
    For every endofunctor $F$ on $\cat{C}$ and every object $X$ in $\cat{C}$, define a transfinite chain of objects $X^\sharp_i$ ($i$ any ordinal) and connecting morphisms 
    \[
        x_{i,j}^\sharp \colon X_i^\sharp \to X_j^\sharp\quad \text{($i \leq j$)}
    \]
    by the following transfinite induction:
    \begin{description}
        \item[First step] 
        $X_0^\sharp = 0$, $X_1^\sharp = F0 + X$ with $x_{0,1}^\sharp$ the unique morphism $0 \xrightarrow{!} F0 + X$,
		
        \item[Isolated step]
        $X_{i+1}^\sharp = FX_i^\sharp + X$ for all ordinals $i$, $x_{i+1, j+1}^\sharp = Fx_{i,j}^\sharp + X$ for all $i \leq j$.
			
        \item[Limit step]
        $X_j^\sharp = \colim_{i < j} X_i^\sharp$ for all limit ordinals $j$ with colimit cocone $x_{i,j}^\sharp$, $i < j$. 
    \end{description}
\end{construction}
The general intuition behind Ad\'amek's construction is that the chain above should converge to the free algebra. As in any adjoint-functor-theorem-like situation, there is no guarantee in full generality that this happens. One says that the free algebra construction \emph{stops} after $k$ steps ($k$ being an ordinal) provided that $x_{k, k + 1}^\sharp$ is an isomorphism.

\begin{theorem}[{\cite[Proposition~3.14]{AdamekPorst2003}}] \label{t:free-algebra-construction}
    If the free-algebra construction stops after $k$ steps, then $X_k^\sharp$ is a free $F$-algebra on $X$.
    In more detail, denoting the components of $x_{k,k+1}^{-1} \colon F X_k^\sharp + X \to X_k \sharp$ by
    \[
        \alpha \colon FX_k^\sharp \to X_k ^\sharp \quad \text{and} \quad \eta_X \colon X \to X_k^\sharp
    \]
    respectively, these form a free $F$-algebra on $X$.
\end{theorem}

We recall from \cite{AdamekTrnkova1990} that a functor $F \colon \cat{C} \to \cat{C}$ with $\cat{C}$ cocomplete is called a \emph{constructive varietor} provided that its free-algebra construction stops for each object $X$ in $\cat{C}$.

\begin{proposition} \label{p:chains}
    If $F$ preserves colimits of $k$-chains ($k$ an infinite limit ordinal), then the free-algebra construction stops after $k$ steps.
\end{proposition}

\begin{proof}
    See \cite[Corollary~3.17 and Remark~3.16]{AdamekPorst2003}.
\end{proof}

\begin{proposition} \label{p:preservation-reflexive-coeq}
    If a varietor $F \colon \cat{C} \to \cat{C}$ preserves colimits of a certain type, so does the algebraically-free monad $F^*$ on $F$.
\end{proposition}

\begin{proof}
    By the same argument in \cite[5.6.5]{riehl2017category}\footnote{The proof is even simpler, because there are fewer conditions to check.}, the forgetful functor $U \colon \Alg(F) \to \cat{C}$ creates the colimits of that type. This finishes the proof. Indeed, calling $L$ the left adjoint to $U$, we have $F^* \cong UL$ by definition, and both functors preserve those colimits.
\end{proof}

\subsection{Monadic functors compose?}
The last categorical prerequisite of the paper concerns another classical problem, i.e.\ whether monadic functors compose. In general, the answer to this question is fairly negative.

\begin{example}
    Any locally presentable category $\cat{K}$ is reflective in (and thus monadic over) a presheaf category $\cat{K} \to \cat{Psh}(\cat{C})$, which is monadic over $\cat{Set}^{\mathrm{Ob}(\cat{C})}$. But if the composition $\cat{K} \to \cat{Psh}(\cat{C}) \to \Set^{\mathrm{Ob}(\cat{C})}$ was monadic, then $\cat{K}$ would be Barr-exact (because categories monadic over $\Set^X$ for $X$ a set are all exact), which is in general not true. A great exemplification of this phenomenon is $\cat{K} = \cat{Cat}$, the category of small categories.
\end{example}

Yet, under some assumptions on the monads, we can indeed infer that the composite of monadic functors is monadic. This relies on a technical analysis of how reflexive coequalizers are constructed in the category of algebras for a monad. We recall that a coreflexive equalizer is an equalizer of a parallel pair $f, g \colon X \to Y$ having a common retraction, i.e.\ a morphism $h \colon B \to A$ such that $h \circ f = h \circ g = 1_X$.

Finally, we can give the following proposition due to \cite{errata-elmendorf1997}. To avoid any confusion with the previous subsection, we call $\cat{C}[S]$ the category of algebras for a monad $S$.

\begin{proposition} \label{p:composite-of-monads}
    Let $S$ be a monad in a category $\cat{C}$ and let $T$ be a monad in the category $\cat{C}[S]$ of $S$-algebras.
    If $T$ preserves reflexive coequalizers in $\cat{C}[S]$, then the category $\cat{C}[S][T]$ of $T$-algebras in $\cat{C}[S]$ is isomorphic to the category $\cat{C}[TS]$ of algebras over the compound monad $TS$ in $\cat{C}$.
    Moreover, the unit of $T$ defines a map $S \to TS$ of monads in $\cat{C}$.
    An analogous assertion holds for comonads. 
\end{proposition}

\begin{theorem} \label{t:composition-monadic}
    Let $T \colon \cat{C} \to \cat{C}$ be a varietor that preserves reflexive coequalizers.
    For every monadic functor $G \colon \cat{C} \to \cat{D}$, the composite $\Alg(T) \xrightarrow{U} \cat{C} \xrightarrow{G} \cat{D}$ is monadic.
\end{theorem}

\begin{proof}
    By \cref{p:preservation-reflexive-coeq}, the algebraically-free functor $F^*$ on $F$ preserves reflexive colimits.
    By \cref{p:composite-of-monads}, the composite $\Alg(T) \xrightarrow{U} \cat{C} \xrightarrow{G} \cat{D}$ is monadic.
\end{proof}

As we previously said, this statement admits an expected dualization to the comonad case.

\section{Vietoris on compact Hausdorff spaces}

Classical modal logic extends classical propositional logic by adding unary operators $\Diamond p$ (usually interpreted as possibility) and $\Box p$ (usually interpreted as necessity), together with appropriate rules.
Just like the algebras of classical propositional logic are Boolean algebras, the algebras of classical modal logic are modal algebras; a modal algebra is a Boolean algebra with a unary operation $\Box$ satisfying $\Box 1 = 1$ and $\Box(x \land y) = \Box x \land \Box y$.
(One can then define $\Diamond x \coloneqq \lnot \Box \lnot x$.)

To represent modal algebras one builds on top of Stone duality \cite{Stone1936}, which states that the category of Boolean algebras and homomorphisms is dually equivalent to the category of Boolean spaces (also known as Stone spaces or profinite spaces)---i.e.\ compact Hausdorff spaces with a basis of closed open sets---and continuous functions.
Building on Stone duality, J\'onsson-Tarski duality states that the category $\MA$ of modal algebras and homomorphisms is dually equivalent to the category of descriptive frames, which are Boolean spaces equipped with a binary relation $R$ (known as ``accessibility relation'') satisfying certain properties.
In its present form it was established by Esakia \cite{Esakia1974} and Goldblatt \cite{Goldblatt1976} (but see also Halmos \cite{Halmos1956}).

One of the properties satisfied by an accessibility relation $R$ is that for each $x \in X$ the forward image $R[\{x\}]$ of $x$ is closed.
So, the relation $R$ can be alternatively described as a function from $X$ to the set $\V X$ of closed subsets of $X$.
The functions $X \to \V X$ arising in this way are precisely those that are continuous with respect to the so-called \emph{Vietoris topology} on $\V X$ \cite{Vietoris1922}.
The Vietoris construction gives rise to an endofunctor $\V_{\cat{BooSp}}$ on the category of Boolean spaces and continuous functions.
It turns out that the category $\MA$ of modal algebras is dually equivalent to the category of coalgebras for $\V_{\cat{BooSp}}$ \cite{Abramsky1988,KupkeKurzEtAl2004,KupkeKurzEtAl2005}.

The axioms of modal algebras are equational, i.e.\ they have the form
\[
\forall x_1\dots\forall x_n\  \tau(x_1, \dots, x_n) = \sigma(x_1, \dots, x_n),
\]
where $\tau$ and $\sigma$ are terms. 
Thus, modal algebras form an equational class (also known as a \emph{variety}) of finitary algebras.
Then, the opposite of the category of coalgebras for $\V_{\cat{BooSp}}$ is a variety of finitary algebras.
In this section, we prove that a similar result holds when we replace Boolean spaces with compact Hausdorff spaces.
We will first recall the definition of the Vietoris functor $\V$ on compact Hausdorff spaces and then we prove that the opposite of the category of coalgebras for $\V$ is monadic over $\Set$ (\cref{t:composition-monadic}), i.e.\ is equivalent to a variety of possibly infinitary algebras.
In this way, we obtain an analogue of J\'onsson-Tarski duality in the larger setting of compact Hausdorff spaces.

We build on the fact that, as for Boolean spaces, the opposite of the category of compact Hausdorff spaces is monadic over $\Set$, as witnessed by the representable functor $\hom_{\CH}(-,[0,1]) \colon \CH^\op \to \Set$.
This fact was observed by Duskin in \cite[{}5.15.3]{Duskin1969} (for a full proof see \cite[Chapter~9, Theorem 1.11]{BarrWells2005}).

\begin{notation}
    We let $\CH$ denote the category of compact Hausdorff spaces and continuous functions.
\end{notation}

\begin{definition}[\cite{Vietoris1922}]
    Given a compact Hausdorff space $X$, we topologize the set $\V X$ of closed subsets of $X$ with the \emph{Vietoris topology}, generated by the sets
    \begin{align*}
        \Box U & \df \{K \in \V X \mid K \subseteq U \} && \text{($U$ open of $X$),}\\
        \Diamond U & \df \{K \in \V X \mid K \cap U \neq \varnothing\} && \text{($U$ open of $X$).}
    \end{align*}
    The space $\V X$ is a compact Hausdorff space \cite{Vietoris1922}, called the \emph{Vietoris hyperspace of $X$}.
\end{definition}

\begin{definition}
    We let $\V \colon \CH \to \CH$ denote the functor that maps
    \begin{itemize}
        \item a compact Hausdorff space $X$ to its Vietoris hyperspace $\V X$.
        \item a morphism $f \colon X \to Y$ to the function $\V f \colon \V X \to \V Y$, $K \mapsto f[K]$.
    \end{itemize}
    We call $\V$ the \emph{Vietoris functor (on compact Hausdorff spaces)}.
\end{definition}

To prove that the opposite of the category of coalgebras for $\V$ is monadic over $\Set$, we first observe that it is monadic over $\CH^\op$ (which in turn is monadic over $\Set$), using the following fact.

\begin{proposition}[{\cite[Corollary 3.37]{HofmannNevesEtAl2019}}] \label{p:preserves-cod-limits}
    The Vietoris functor $\V \colon \CH \to \CH$ preserves codirected limits. 
\end{proposition}

\begin{corollary} \label{c:monadic}
    The Vietoris functor $\V  \colon \CH \to \CH$ is a covarietor.
\end{corollary}

\begin{proof}
    By \cref{p:preserves-cod-limits}, $\V$ preserves $\omega^\op$-limits. Therefore, by \cref{p:chains,t:free-algebra-construction}, $\V$ is a covarietor.
\end{proof}

Let $U \colon \CoAlg(\V ) \to \CH$ be the forgetful functor that maps 
\begin{enumerate}
    \item a coalgebra $f \colon X \to \V X$ to the space $X$,
    \item a morphism $h \colon X \to Y$ (from $f \colon X \to \V X$ to $g \colon Y \to \V Y$) to $h$ itself.
\end{enumerate}

\begin{corollary}
    The forgetful functor $U \colon \CoAlg(\V ) \to \CH$ is comonadic.
\end{corollary}

As we had discussed before, we aim to show that $\CoAlg(\V)^\op$ is monadic over $\Set$.
We will do it by showing that the composite of the two monadic functors $U^\op \colon \CoAlg(\V)^\op \to \CH^\op$ and $\hom_\CH(-,[0,1]) \colon \CH^\op \to \Set$ is monadic.
In order to do so we shall verify that the hypotheses of \cref{t:composition-monadic} from the previous section are verified, i.e.\ that $\V \colon \CH \to \CH$ preserves coreflexive equalizers.
In \cite{TownsendVickers2014}, C.\ Townsend and S.\ Vickers prove that the lower powerlocale functor, the upper powerlocale functor and the Vietoris powerlocale functor preserve coreflexive equalizers (respectively, Propositions 66, 68 and 70 in \cite{TownsendVickers2014}).
We provide a point-based proof in the case of compact Hausdorff spaces.

\begin{proposition} \label{p:Vietoris-refl-coeq}
    The Vietoris functor $\V  \colon \CH \to \CH$ preserves coreflexive equalizers.
\end{proposition}

\begin{proof}
    Let $h \colon E \to X$ be an equalizer of two morphisms $f,g \colon X \rightrightarrows Y$ in $\CH$ with a common retraction, and let us prove that $\V h$ is an equalizer of $\V f$ and $\V g$.
    Since the forgetful functor $\CH \to \Set$ preserves and reflects equalizers, it is enough to prove that (the underlying function of) $\V  h$ is the equalizer in $\Set$ of (the underlying functions of) $\V  f$ and $\V  g$.
    By functoriality of $\V $, we have $\V  f \circ \V  h = \V  g \circ \V  h$.
    The function $\V  h$ is injective because $h$ is injective.
    Let $K \in \V  X$ be such that $(\V f)(K) = (\V g) (K)$, i.e.\ $f[K] = g[K]$.
    We should prove that $K$ belongs to the image of $\V h$.
    Since $f[K] = g[K]$, for every $x \in K$ there is $x' \in K$ such that $f(x) = g(x')$.
    Since $f$ and $g$ have a common retraction $k \colon Y \to X$, we have $x = kf(x) = kg(x') = x'$, and so $f(x) = g(x') = g(x)$.
    Therefore, for every $x \in K$ we have $f(x) = g(x)$.
    Thus, $K \subseteq \ima h$, and hence $K = h[h^{-1}[K]]$, i.e.\ $(\V h)(h^{-1}[K]) = K$.
    Therefore, $K$ belongs to the image of $\V  h$.
    Thus, $\V h$ is the equalizer of $\V f$ and $\V g$ in $\Set$, and hence also in $\CH$.
\end{proof}

\begin{theorem} \label{t:composition-is-monadic-Vietoris}
    Let $G$ be a comonadic functor from $\CH$ to a category $\cat{C}$.
    The composite $\CoAlg(\V ) \xrightarrow{U} \CH \xrightarrow{G} \cat{C}$ is comonadic. 
\end{theorem}

\begin{proof}
    By \cref{c:monadic}, $\V$ is a covarietor.
    By \cref{p:Vietoris-refl-coeq}, $\V$ preserves coreflexive equalizers.
    By \cref{t:composition-monadic}, the composite $\CoAlg(\V) \xrightarrow{U} \CH \xrightarrow{G} \cat{C}$ is comonadic.
\end{proof}

\begin{theorem} \label{t:monadicity-vietoris}
    $\CoAlg(\V)^\op$ is monadic over $\Set$.
\end{theorem}

\begin{proof}
    By \cite[5.15.3]{Duskin1969}, the representable functor 
    \[
    \hom_\CH(-,[0,1]) \colon \CH^\op \to \Set
    \]
    is monadic.
    Then, by \cref{t:composition-is-monadic-Vietoris}, the composite functor below is monadic. 
    \[
    \CoAlg(\V )^\op \xrightarrow{U^\op} \CH^\op \xrightarrow{G} \Set \qedhere
    \]
\end{proof}

\begin{remark} \label{rem:cocomBarr}
    From the monadicity result in \cref{t:monadicity-vietoris} one can deduce various properties of $\CoAlg(\V)^\op$, such as its (co)completeness and Barr-exactness.
\end{remark}

\begin{remark}
    From the proof of \cref{t:monadicity-vietoris} we can extract the description of a monadic functor $\CoAlg(\V )^\op \to \Set$, as follows.
    To an object $f \colon X \to \V (X)$ of $\CoAlg(\V )$ we associate the set $\hom_\CH(X, [0,1])$.
    To a morphism $g \colon X_1 \to X_2$ in $\CoAlg(\V )$ from $f_1 \colon X_1 \to \V (X_1)$ to $f_2 \colon X_2 \to \V (X_2)$ we associate the function
    \begin{align*}
        \hom_\CH(X_2, [0,1]) & \longrightarrow \hom_\CH(X_1, [0,1])\\
        h & \longmapsto h \circ g.
    \end{align*}
    The unit interval can be replaced by any regular injective regular cogenerator object of $\CH$.
\end{remark}

\section{Convex Vietoris on compact ordered spaces}

In the previous section we proved that the opposite of the category of coalgebras for the Vietoris functor $\V$ on compact Hausdorff spaces is monadic over $\Set$.
In this and the next section we establish analogous results where we replace the Vietoris construction with some of its variants.
In this section, we consider the convex Vietoris hyperspace (which corresponds to the Plotkin powerdomain in domain theory).
This construction is defined on Nachbin's compact ordered spaces, which are an ordered version of compact Hausdorff spaces, and restricts to the classical Vietoris construction on those compact ordered spaces with a trivial order.
In this section, we prove that the opposite of the category of coalgebras for the convex Vietoris functor on compact ordered spaces is monadic over $\Set$.

This last result can be understood as an analogue of the J\'onsson-Tarski duality for positive modal algebras.
Positive modal logic is, roughly speaking, modal logic without negation.
It was introduced by Dunn \cite{Dunn1995}, and it is the restriction of the modal local consequence relation defined by the class of all Kripke models to the propositional modal language whose connectives are $\wedge$, $\lor$, $\top$, $\bot$, $\Box$, $\Diamond$.
The algebras of positive modal logic are called positive modal algebras \cite{Jansana2002}, and are bounded distributive lattices with $\Box$ and $\Diamond$ and some equational axioms.

To represent positive modal algebras one builds on Priestley duality, which states that the category of bounded distributive lattices is dually equivalent to the category of Priestley spaces, i.e.\ Boolean spaces with a partial order and appropriate axioms \cite{Priestley1970}.
Building on Priestley duality, the category of positive modal algebras is dually equivalent to the category of coalgebras for the convex Vietoris functor on Priestley spaces \cite{Palmigiano2004,BezhanishviliKurz2007,BonsangueKurzEtAl2007,VenemaVosmaer2014}.
We extend this equivalence from Priestley spaces to compact ordered spaces maintaining the algebraicity of the algebraic side.

\begin{definition}[{\cite{Nachbin1948}, \cite[p.~44]{Nachbin1965}}]
    A \emph{compact ordered space} is a compact Hausdorff space $X$ equipped with a partial order that is closed in the product topology of $X \times X$. We let $\CompOrd$ denote the category of compact ordered spaces and continuous order-preserving maps.
\end{definition}

The study of compact ordered spaces originated in Nachbin's classic book \cite{Nachbin1965}; see also \cite[Section~VI-6]{GierzHofmannEtAl2003} and \cite{Tholen2009}.
Other names for compact ordered spaces are ``ordered compact spaces'', ``partially ordered compact spaces'', ``separated ordered compact spaces'', ``compact pospaces'', and ``Nachbin spaces''.

In the next section we will recall their close connection with stably compact spaces, their topological alter ego.
In this section, we turn our attention to the convex Vietoris functor on compact ordered spaces.
One appealing property of this functor is that it restricts to the usual Vietoris functor on $\CH$.
We prove that also the opposite of the category of coalgebras for this functor is monadic over $\Set$.

\begin{remark} \label{r:dual}
    Dualising the order of a compact ordered space defines a compact ordered space, as well.
\end{remark}

\begin{notation}
    An upward (resp.\ downward) closed subset of a poset will also be called an \emph{upset} (resp.\ \emph{downset}).
    A subset $Y$ of a poset is said to be \emph{convex} if $Y \ni x \leq y \leq z \in Y$ implies $y \in Y$.
    We let $\u Y$ and $\d Y$ denote respectively the up-closure and down-closure of a subset $Y$ of a poset $X$.
    We use $\u x$ and $\d x$ as shorthands for $\u \{x\}$ and $\d \{x\}$.
    We denote the smallest convex set containing a set $Y$ by
    \[
    \ud Y \coloneqq \{x \in X \mid \exists y,y' \in Y \text{ s.t.\ }y \leq x \leq y'\} = \u Y \cap \d Y.
    \]
\end{notation}

In the following two lemmas we recall some basic properties of compact ordered spaces.

\begin{lemma}[{\cite[Proposition 4, p.~44]{Nachbin1965}}] \label{l:up-set-closed}
    Let $X$ be a compact ordered space.
    For every closed subset $K$ of $X$, the sets $\d K$ and $\u K$ are closed.
\end{lemma}

\begin{lemma}[{\cite[Theorem~4, p.\ 48]{Nachbin1965}}] \label{l:normal}
    Let $X$ be a compact ordered space, and let $K$ be a closed downset of $X$ and $L$ a closed upset of $X$ such that $K \cap L = \varnothing$.
    There are an open downset $U$ of $X$ and an open upset $V$ of $X$ such that $K \subseteq U$, $L \subseteq V$ and $U \cap V = \varnothing$.
\end{lemma}

In the following, we will define the convex Vietoris functor on compact ordered spaces.
In the case of a Priestley space $X$, its convex Vietoris hyperspace was described in \cite[Sections~3.3 and 3.4]{Palmigiano2004} as a quotient of the classical Vietoris hyperspace $\V X$.
In \cite{BezhanishviliKurz2007,BonsangueKurzEtAl2007,VenemaVosmaer2014}, an alternative equivalent description was given. In this description, a Priestley space $X$ is mapped to a Priestley space whose underlying set is the set of compact convex subsets of $X$.
The equivalence of these two definitions is proved in \cite[Theorem~4.8]{BezhanishviliHardingEtAl2023}.
We refer to \cite[Section~3.1]{Lauridsen2015} for a detailed proof of the fact that this second construction is indeed a well-defined functor on the category of Priestley spaces.

In the following, we define the convex Vietoris hyperspace construction and show that this defines an endofunctor on the category of compact ordered spaces.\footnote{This construction corresponds to a well-known construction in the equivalent context of stably compact spaces, based on the notion of so-called \emph{lenses}; see e.g.\ \cite[Section~6]{Lawson2011} (where however the empty set is excluded from the hyperspace). To the best of our knowledge, the first account of this construction is in \cite{Johnstone1985} (see in particular Corollary 3.10 therein for the spatial setting).}

\begin{definition}
    We let $\Vc X$ denote the set of closed convex subsets of a compact ordered space $X$.
    We equip $\Vc X$ with the topology generated by the sets
    \begin{align*}
    	\Diamond U & \df \{K \in \Vc X \mid K \cap U \neq \varnothing\} &&\text{($U$ open upset or open downset of $X$),}\\
    	\Box U & \df \{K \in \Vc X \mid K \subseteq U \} && \text{($U$ open upset or open downset of $X$).} 
    \end{align*}
    We equip $\Vc X$ with the Egli-Milner order, i.e.\ for $K, L \in \Vc X$, we set 
    \[
    K \leq_{\EM} L \iff \u L \subseteq \u K \text{ and } \d K \subseteq \d L.
    \]
    Explicitly, this condition amounts to
    \[
        \forall y \in L\ \exists x \in K \text{ s.t.\ } x \leq y \quad \text{ and }\quad
	\forall x \in K\ \exists y \in L \text{ s.t.\ } x \leq y.
    \]
    We call $\Vc X$ the \emph{convex Vietoris hyperspace of $X$}.
\end{definition}

\begin{theorem} \label{t:Vc-is-compord}
    For every compact ordered space $X$, the convex Vietoris hyperspace $\Vc X$ of $X$ is a compact ordered space.
\end{theorem}

\begin{proof}
    To prove that $\Vc X$ is Hausdorff, let $K, L \in \Vc$ be distinct.
    Without loss of generality, we may suppose that there is $x \in K \setminus L$.
    Moreover, by convexity of $L$, either $L$ is disjoint from $\d x$ or from $\u x$.
    Without loss of generality, we may suppose that $L$ is disjoint from $\u x$ (the other case being similar).
    By \cref{l:up-set-closed,l:normal}, there are an open upset $U$ of $X$ and an open downset $V$ of $X$ such that $\d x \subseteq U$, $L \subseteq V$, and $U \cap V = \varnothing$.
    Then $K \in \Diamond U$ because $x \in K \cap U$, $L \in \Box V$ because $L \subseteq V$, and $\Diamond U$ and $\Box V$ are disjoint because $U$ and $V$ are disjoint.
    This proves that $\Vc X$ is Hausdorff.

    We prove compactness.
    By the Alexander subbase theorem, it is enough to prove that every cover of $\Vc X$ by subbasic open sets (i.e.\ by boxes and diamonds of open upsets and open downsets) has a finite subcover.
    Suppose
    \begin{equation} \label{eq:cover}
        \Vc X = \bigcup_{i \in I} \Box U_i \cup \bigcup_{j \in J} \Diamond U'_j \cup \bigcup_{k \in K} \Box V_k \cup \bigcup_{l \in L} \Diamond V'_l,
    \end{equation}
    where $U_i$ and $U_j$ are open upsets for all $i \in I$ and $j \in J$, and $V_k$ and $V_l$ are open downsets for all $k \in K$ and $l \in L$.
    Set $W \coloneqq \bigcup_{j \in J} U'_j \cup \bigcup_{l \in L}V'_l$, and $F \coloneqq X \setminus W$.
    Since $W$ is a union of open sets, it is open, and its complement $F$ is closed.
    Since $W$ is a union of an upset and a downset, its complement $F$ is an intersection of a downset and an upset; it follows that $F$ is convex.
    Therefore, $F \in \Vc X$.
    
    By \eqref{eq:cover}, there is $i \in I$ such that $F \in \Box U_i$, there is $j \in J$ such that $F \in \Diamond U'_j$, there is $k \in K$ such that $F \in \Box V_k$ or there is $l \in L$ such that $F \in \Diamond V'_l$.
    Without loss of generality, we can assume that either there is $i \in I$ such that $F \in \Box U_i$ or there is $j \in J$ such that $F \in \Diamond U'_j$ (the other cases being similar).
    We can exclude the case that there is $j \in J$ such that $F \in \Diamond U'_j$ because, since $F = X \setminus (\bigcup_{j \in J} U'_j \cup \bigcup_{l \in L}V'_l)$, $F$ is disjoint from $U'_j$ for every $j \in J$.
    Therefore, there is $i_0 \in I$ such that $F \in \Box U_{i_0}$, i.e.\ $F \subseteq U_{i_0}$.

    We have
    \[
    X = F \cup W \subseteq U_{i_0} \cup W = U_{i_0} \cup \bigcup_{j \in J} U'_j \cup \bigcup_{l \in L}V'_l.
    \]
    and therefore
    \[
    X = U_{i_0} \cup \bigcup_{j \in J} U'_j \cup \bigcup_{l \in L}V'_l.
    \]
    Since $X$ is compact, there are a finite $J' \subseteq J$ and a finite $L' \subseteq L$ such that
    \[
    X = U_{i_0} \cup \bigcup_{j \in J'} U'_j \cup \bigcup_{l \in L'}V'_l.
    \]
    It follows that
    \[
    \Vc X = \Box U_{i_0} \cup \bigcup_{j \in J'} \Diamond U'_j \cup \bigcup_{l \in L'} \Diamond V'_l.
    \]
    This proves compactness.

    We prove that the Egli-Milner order is a closed subset of $(\Vc X) \times (\Vc X)$.
    To do so, we prove that its complement is open.
    Let $(K, L) \in ((\Vc X) \times (\Vc X)) \setminus ({\leq_\EM})$.
    Either $\u L \nsubseteq \u K$ or $\d K \nsubseteq \d L$.
    Without loss of generality, we can suppose $\u L \nsubseteq \u K$, the other case being similar.
    Therefore, there is $x \in L \setminus \u K$.
    From $x \notin \u K$ we deduce $\d x \cap \u K = \varnothing$.
    By \cref{l:up-set-closed,l:normal}, there are an open downset $U$ of $X$ and an open upset $V$ of $X$ such that $\d x \subseteq U$, $K \subseteq V$ and $U \cap V = \varnothing$.
    The set $\Diamond U$ is an open neighbourhood of $L$ in $\Vc X$ (because $x \in U \cap L$), $\Box V$ is an open neighbourhood of $K$ in $\Vc$ (because $K \subseteq U$).
    Moreover, for every $L' \in \Diamond U$ and every $K' \in \Box V$, we have $L' \nsubseteq \u K$, which implies $K' \nleq_\EM L'$.
    Therefore, $(\Box V) \times (\Diamond U)$ is an open neighbourhood of $(K, L)$ disjoint from the Egli-Milner order of $\Vc X$.
\end{proof}

\begin{definition}
    We let $\Vc \colon \CompOrd \to \CompOrd$ denote the functor that maps
    \begin{itemize}
        \item a compact ordered space $X$ to its convex Vietoris hyperspace $\Vc(X)$.
        \item a morphism $f \colon X \to Y$ to the morphism $\Vc f \colon \Vc X \to \Vc Y$ that maps an element $K \in \Vc X$ to the convex closure $\ud f[K]$ of $f[K]$.
    \end{itemize}
    We call $\Vc$ the \emph{convex Vietoris functor}.
\end{definition}

\begin{lemma}
	$\Vc$ is a well-defined functor. 
\end{lemma}

\begin{proof}
	\cref{t:Vc-is-compord} shows that $\Vc$ is well-defined on objects.
	
	We prove that $\Vc$ is well-defined on morphisms.
	Let $f \colon X \to Y$ be a morphism of compact ordered spaces.
	For every $K \in \Vc X$ the set $f[K]$ is compact and thus its convex closure $\ud f[K]$ is also compact (by \cref{l:up-set-closed}, since $\ud f[K] = \u K \cap \d K$).
	Therefore, $\Vc f \colon \Vc X \to \Vc Y$ is a well-defined function.
	
	We prove that $\Vc f$ is continuous. Let $U$ be an open upset of $\Vc Y$.
	We have
	\begin{align*}
		(\Vc f)^{-1}[\Diamond U] & = \{K \in \Vc X \mid (\Vc f)(K) \in \Diamond U\}\\
		& =  \{K \in \Vc X \mid \ud f[K] \in \Diamond U\}\\
		& = \{K \in \Vc X \mid \ud f[K] \cap U \neq \varnothing \}\\
		& = \{K \in \Vc X \mid f[K] \cap U \neq \varnothing \}\\
		& = \{K \in \Vc X \mid K \cap f^{-1} [U] \neq \varnothing \}\\
		& = \Diamond [f^{-1} [U]],
	\end{align*}
	which is an open subset of $\Vc X$.
	Moreover,
	\begin{align*}
		(\Vc f)^{-1}[\Box U] & = \{K \in \Vc X \mid (\Vc f)(K) \in \Box U\}\\
		& =  \{K \in \Vc X \mid \ud f[K] \in \Box U\}\\
		& = \{K \in \Vc X \mid \ud f[K] \subseteq U \}\\
		& = \{K \in \Vc X \mid f[K] \subseteq U \}\\
		& = \{K \in \Vc X \mid K \subseteq f^{-1} [U] \}\\
		& = \Box [f^{-1} [U]],
	\end{align*}
	which is an open subset of $\Vc X$.
	Analogous facts are analogously proved for $U$ an open downset.
	Therefore, $\Vc f$ is continuous.
		
	We prove that $\Vc f$ is order-preserving. Let $K, L \in \Vc X$ be such that $K \leq_\EM L$, i.e.\ $\u L \subseteq \u K$ and $\d K \subseteq \d L$.
	Then,
	\[
		\u \ud f[L] = \u f[L] = \u f[\u L] \subseteq \u f[\u K] = \u f[K] = \u \ud f[K].
	\]
	This proves $\u ((\Vc f) (L)) \subseteq \u ((\Vc f) (K))$.
	Similarly, $\d ((\Vc f) (K)) \subseteq \d ((\Vc f) (L))$.
	Thus, $(\Vc f) (K) \leq_\EM (\Vc f) (L)$. Therefore, $\Vc f$ is order-preserving.
	
	This proves that $\Vc$ is well-defined on morphisms.
	
	We prove that $\Vc$ preserves composition. Let $f \colon X \to Y$ and $g \colon Y \to Z$ be two morphisms.
	Then, for every $K \in \Vc X$, we have 
	\begin{align*}
		(\Vc (g \circ f))(K) & = \ud (g \circ f)[K] \\
		& = \ud g[f[K]] \\
		& = \ud g [\ud f [K]]\\
		& = \Vc (g) (\Vc (f)(K))\\
		& = (\Vc (g) \circ \Vc (f))(K).
	\end{align*}
	This proves that $\Vc$ preserves composition.
	
	The proof that $\Vc$ preserves identities is straightforward.
\end{proof}

We will prove that $\CoAlg(\Vc)^\op$ is monadic over $\Set$; to this purpose, in view of an application of \cref{t:monadicity-vietoris}, we show that $\Vc$ preserves codirected limits and coreflexive equalizers.

To prove that $\Vc$ preserves codirected limits, we take inspiration from \cite[Section~3]{HofmannNevesEtAl2019}, which proves the same property for the lower Vietoris functor.

\begin{remark} \label{r:limits}
    For details about limits in $\CompOrd$ we refer to \cite[Section~3]{HofmannNora2015}.
    We will use the fact that a cone in $\CompOrd$ is a limit cone if and only if it is a limit cone in $\Set$ (or, equivalently, in $\CH$) and the order is initial.
\end{remark}

By \cite[Remark~4.3]{HofmannNora2023}, $\CompOrd$ inherits a nice characterisation of codirected limits from the category $\CH$.
A first hint of the characterisation of codirected limits in $\CH$ is in Bourbaki \cite{Bourbaki1942}.
The characterisation was proved in \cite{Hofmann1999} (see also \cite[Theorem~3.29]{HofmannNevesEtAl2019}), and it is called the \emph{Bourbaki-criterion}.
Here, we formulate it in the setting of compact ordered spaces.

\begin{theorem}[Bourbaki-criterion for compact ordered spaces, {\cite[Remark~4.3]{HofmannNora2023}}] \label{t:Bourbaki-compord}
    Let $D \colon \cat{I} \to \CompOrd$ be a codirected diagram.
    A cone $(f_i \colon X \to D(i))_{i \in \cat{I}}$ for $D$ is a limit cone if and only if the following conditions hold.
    \begin{enumerate}
        \item \label{i:separation} 
        For all $x, y \in X$, $x \leq y$ if and only if for all $i \in \cat{I}$ we have $f_i(x) \leq f_i(y)$.
        
        \item \label{i:images} 
        For all $i \in \cat{I}$, the image of $f_i$ coincides with the intersection of the images of all $D(j \to i)$; in symbols,
        \[
            \bigcap_{j \to i}\ima D(j \to i) = \ima {f_i}.
        \]
    \end{enumerate}
\end{theorem}

\begin{lemma}\label{i:initial-mono}
    Let $D \colon \cat{I} \to \CompOrd$ be a codirected diagram, let $(f_i \colon X \to D(i))_{i \in \cat{I}}$ be a limit for $D$, and let $K$ and $L$ be closed subsets of $X$.
    \begin{enumerate}
        \item \label{i:upper-inclusion}
        $\u K \subseteq \u L$ if and only if for all $i \in \cat{I}$ we have $\u f_i[K] \subseteq \u f_i[L]$.
        
        \item \label{i:lower-inclusion}
        $\d K \subseteq \d L$ if and only if for all $i \in \cat{I}$ we have $\d f_i[K] \subseteq \d f_i[L]$.	

        \item \label{i:EM-inclusion}
        $\ud K \leq_{\EM} \ud L$ if and only if for all $i \in \cat{I}$ we have $\ud f_i[K] \leq_{\EM} \ud f_i[L]$.
    \end{enumerate}
\end{lemma}

\begin{proof}
    \eqref{i:upper-inclusion}. 
    The left-to-right inclusion holds because $\u K \subseteq \u L$ implies $\u f_i[K]= \u f_i[\u K] \subseteq \u f_i[\u L]  = \u f_i[L]$.
    
    For the right-to-left inclusion, suppose that for all $i \in \cat{I}$ we have $\u f_i[K] \subseteq \u f_i[L]$.
    Let $x \in \u K$.
    Then for all $i \in \cat{I}$ we have $f_i(x) \in \u f_i[K] \subseteq \u f_i[L]$, and thus $f_i(x) \in \u f_i[L]$, which means $f_i[L] \cap \d f_i(x) \neq \varnothing$, which implies $L \cap f_i^{-1}[\d f_i(x)] \neq \varnothing$.
    By compactness, it follows that $L \cap \bigcap_{i \in \cat{I}}f_i^{-1}[\d f_i(x)] \neq \varnothing$.
    By \cref{t:Bourbaki-compord}\eqref{i:separation}, we have $\bigcap_{i \in \cat{I}}f_i^{-1}[\d f_i(x)] = \d x$, and thus $L \cap \d x \neq \varnothing$, which implies $x \in \u L$.
    Since this holds for all $x \in \u K$, we have $\u K \subseteq \u L$.

    \eqref{i:lower-inclusion}.
    This is dual to \eqref{i:upper-inclusion}, in the sense of \cref{r:dual}.

    \eqref{i:EM-inclusion}. We have
    \begin{align*}
        &\ud K \leq_{\EM} \ud L \\
        & \iff \u L \subseteq \u K \text{ and } \d K \subseteq \d L\\
        & \iff \forall i \in \cat{I}\ \u f_i[L] \subseteq \u f_i[K] \text{ and } \d f_i[K] \subseteq \d f_i[L] && \text{by \eqref{i:lower-inclusion} and \eqref{i:upper-inclusion}}\\ 
        & \iff \forall i \in \cat{I}\ \u \ud f_i[L] \subseteq \u f_i[K] \text{ and } \d \ud f_i[K] \subseteq \d f_i[L]\\
        & \iff \forall i \in \cat{I}\ \ud f_i[K] \leq_{\EM} \ud f_i[L]. &&\qedhere
    \end{align*}
\end{proof}

\begin{lemma} \label{l:exchange}
    Let $\mathcal{F}$ be a codirected set of closed subsets of a compact ordered space.
    \begin{enumerate}
        \item \label{i:u-cap}
        $\u \bigcap_{K \in \mathcal{F}} K = \bigcap_{K \in \mathcal{F}} \u K$.
        
        \item \label{i:d-cap}
        $\d \bigcap_{K \in \mathcal{F}} K = \bigcap_{K \in \mathcal{F}} \d K$.
        
        \item \label{i:ud-cap}
        $\ud \bigcap_{K \in \mathcal{F}} K = \bigcap_{K \in \mathcal{F}} \ud K$.
    \end{enumerate}
\end{lemma}

\begin{proof}
    \eqref{i:d-cap} is \cite[Proposition~3.31]{HofmannNevesEtAl2019}.
    \eqref{i:u-cap} is dual, in the sense of \cref{r:dual}.
    \eqref{i:ud-cap} follows from \eqref{i:d-cap} and \eqref{i:u-cap}.
\end{proof}

\begin{lemma}\label{l:image-large}
    Let $D \colon \cat{I} \to \CompOrd$ be a codirected diagram, let $(f_j \colon X \to D(j))_{j \in \cat{I}}$ be a limit for $D$, and let $i \in \cat{I}$.
    \begin{enumerate}
        \item \label{i:bigcap-u}
        $\bigcap_{j \to i} \u \ima D(j \to i) = \u \ima {f_i}$.
        
        \item \label{i:bigcap-d}
        $\bigcap_{j \to i} \d \ima D(j \to i) = \d \ima {f_i}$.
        
        \item \label{i:bigcap-ud}
        $\bigcap_{j \to i} \ud \ima D(j \to i) = \ud \ima {f_i}$.
    \end{enumerate}
\end{lemma}

\begin{proof}
    \eqref{i:bigcap-d}.
    We have
    \begin{align*}
        \bigcap_{j \to i} \u \ima D(j \to i) & = \u \bigcap_{j \to i} \ima D(j \to i) && \text{by \cref{l:exchange}\eqref{i:u-cap}}\\
        & = \u \ima f_i && \text{by \cref{t:Bourbaki-compord}.}
    \end{align*}

    \eqref{i:bigcap-u} and \eqref{i:bigcap-ud} are proved similarly.
\end{proof}

\begin{proposition} \label{p:convex preserves-codirected}
    The convex Vietoris functor $\Vc$ on compact ordered spaces preserves codirected limits.
\end{proposition}

\begin{proof}
    Let $(f_i \colon X \to D(i))_{i \in \cat{I}}$ be a limit for a codirected diagram $D \colon \cat{I} \to \CompOrd$.
    We prove that $(\Vc f_i \colon \Vc X \to \Vc D(i))_{i \in \cat{I}}$ is a limit for $\Vc \circ D$ by verifying that it satisfies the two conditions in the Bourbaki-criterion (\cref{t:Bourbaki-compord}).
    Let $d_{ji}$ denote the morphism $D(j \to i)$.
    
    For the first condition, let $K, L \in \Vc X$ be such that, for all $i \in \cat{I}$, we have $(\Vc f_i)(K) \leq_{\EM} (\Vc f_i)(L)$, i.e.\ $\ud f_i[K] \leq_{\EM} \ud f_i[L]$.
    Then, by \cref{i:initial-mono}\eqref{i:EM-inclusion}, $\ud K \leq_{\EM} \ud L$.
    Therefore, the cone $(\Vc f_i \colon \Vc X \to \Vc D(i))_{i \in \cat{I}}$ satisfies condition \eqref{i:separation} in the Bourbaki-criterion.

    For the second condition, let $i \in \cat{I}$ and let us prove that $\bigcap_{j \to i}\ima \Vc d_{ji} = \ima \Vc{f_{i}}$.
    The right-to-left inclusion is easy.
    For the left-to-right inclusion, let $K \in \bigcap_{j \to i}\ima \Vc d_{ji}$.
    Then, for each $j \to i$ there is a closed convex subset $S_j$ of $D(j)$ such that $\ud d_{ji}[S_j] = K$.
    From $\ud d_{ji}[S_j] = K$ and the convexity of $K$, we deduce $S_j \subseteq K$.
    Since we also have $d_{ji}[S_j] \subseteq \ima d_{ji}$, we have $d_{ji}[S_j] \subseteq K \cap \ima d_{ji}$.
    Therefore, $K = \ud d_{ji}[S_j] \subseteq \ud (K \cap \ima d_{ji})$.
    Since this holds for all $j \to i$, we have $K \subseteq \bigcap_{j \to i} \ud (K \cap \ima d_{ji})$.
    Therefore, 
    \begin{align*}
        K & = K \cap \bigcap_{j \to i} \ud (K \cap \ima d_{ji})\\
        & = K \cap \ud \bigcap_{j \to i} (K \cap \ima d_{ji}) && \text{by \cref{l:exchange}\eqref{i:ud-cap}}\\
        & = \ud \left(K \cap \bigcap_{j \to i} \ima d_{ji}\right) && \text{using the convexity of $K$}\\
        & = \ud (K \cap \ima f_i) && \text{by \cref{t:Bourbaki-compord}}\\
        & = \ud f_i[f_i^{-1}[K]].
    \end{align*}
    Thus, $K = (\Vc f_{i})(f_i^{-1}[K])$, and therefore $K \in \ima (\Vc f_{i})$.
\end{proof}

\begin{proposition} \label{p:compord-monadic}
    The convex Vietoris functor $\Vc  \colon \CompOrd \to \CompOrd$ is a covarietor.
\end{proposition}

\begin{proof}	
    By \cref{p:convex preserves-codirected}, $\Vc$ preserves codirected limits: in particular, it preserves $\omega^\op$-limits.
    Thus, by \cref{p:chains,t:free-algebra-construction}, $\Vc$ is a covarietor.
\end{proof}

\begin{corollary}
    The forgetful functor $\CoAlg(\Vc) \to \CompOrd$ is comonadic.
\end{corollary}

In light of \cref{t:composition-monadic}, the next step to prove that $\CoAlg(\Vc)^\op$ is monadic over $\Set$ is to prove that $\Vc$ preserves coreflexive equalizers.

\begin{lemma}[{\cite[Theorem~2.6]{HofmannNevesEtAl2018}}]\label{l:regmono-in-compord}
    The regular monomorphisms in $\CompOrd$ are precisely the order-reflecting morphisms.
\end{lemma}

In the following, $\leq_{\EM}$ denotes the Egli-Milner order, which can be defined on the set of convex subsets of a poset in the same way:
\[
K \leq_{\EM} L \iff \u L \subseteq \u K \text{ and } \d K \subseteq \d L.
\]

\begin{lemma} \label{l:reg-mono}
    Let $f \colon X \to Y$ be an order-reflecting map between posets, and let $K$ and $L$ be subsets of $X$.
    \begin{enumerate}
        \item \label{i:reg-mono-ks}
        $\u K \subseteq \u L$ if and only if $\u f[K] \subseteq \u f[L]$.
        \item	\label{i:reg-mono-c}
        $\d K \subseteq \d L$ if and only if $\d f[K] \subseteq \d f[L]$.
        \item \label{i:reg-mono-ud}
        $\ud K \leq_{\EM} \ud L$ if and only if $\ud f[K] \leq_{\EM} \ud f[L]$.
    \end{enumerate}
\end{lemma}

\begin{proof}
    \eqref{i:reg-mono-ks}.
    The left-to-right implication holds because $\u K \subseteq \u L$ implies $\u f[K] = \u f[\u K] \subseteq \u f[\u L] = \u f[L]$.
    We prove the opposite implication.
    Since $f$ is order-reflecting, for every upset $Y$ of $X$ we have $f^{-1}[\u f[Y]] = Y$.
    Therefore, if $\u f[K] \subseteq \u f[L]$, then $\u K = \u f^{-1}[\u f[K]] \subseteq \u f^{-1}[\u f[L]] = \u L$.

    \eqref{i:reg-mono-c}. This is the dual of \eqref{i:reg-mono-ks}, in the sense of \cref{r:dual}.

    \eqref{i:reg-mono-ud}. We have 
    \begin{align*}
        K \leq_{\EM} L & \iff \u L \subseteq \u K \text{ and } \d K \subseteq \d L\\
        &  \iff \u f[L] \subseteq \u f[K] \text{ and } \d f[K]  \subseteq \d f[L] && \text{by \eqref{i:reg-mono-ks} and \eqref{i:reg-mono-c}}\\
        & \iff \u \ud f[L] \subseteq \u \ud f[K] \text{ and }\d \ud f[K] \subseteq \d \ud f[L]\\
        & \iff \ud f[K] \leq_{\EM} \ud f[L].&&\qedhere
    \end{align*}
\end{proof}

\begin{lemma} \label{l:preserve-regmono}
    The convex Vietoris functor $\Vc \colon \CompOrd \to \CompOrd$ preserves regular monomorphisms, i.e.,
    if $f$ is an order-reflecting morphism of compact ordered spaces, then $\Vc f$ is order-reflecting.
\end{lemma}
\begin{proof}
    We recall from \cref{l:regmono-in-compord} that regular monomorphisms in $\CompOrd$ are precisely the order-reflecting morphisms.
    Let $f \colon X \to Y$ be a regular monomorphism in $\CompOrd$.
    For $K, L \in \Vc X$ we have
    \begin{align*}
        &\Vc f (K) \leq_{\EM} \Vc f (L) \\
        & \iff \ud f[K] \leq_{\EM} \ud f[L]\\
        & \iff K \leq_{\EM} L && \text{by \cref{l:reg-mono}\eqref{i:reg-mono-ud}.}
    \end{align*}
    This shows that $\Vc f$ is order-reflecting, and hence a regular monomorphism.
\end{proof}

\begin{remark} \label{r:equalizer-compord}
    By \cref{r:limits,l:regmono-in-compord}, a morphism $h \colon Z \to X$ in $\CompOrd$ is the equalizer of $f, g \colon X \rightrightarrows Y$ if and only if (i) an element $x \in X$ belongs to the image of $h$ if and only if $f(x) = g(x)$, and (ii) $h$ is order-reflecting.
\end{remark}

\begin{lemma} \label{l:exist-dirsup}
    Let $X$ be a compact ordered space.
    \begin{enumerate}
        \item \label{i:dir-1}
        Any directed subset $Y$ of $X$ has a supremum in $X$, which belongs to the topological closure of $Y$.
        
        \item \label{i:codir-1}
        Any codirected subset $Y$ of $X$ has an infimum in $X$, which belongs to the topological closure of $Y$.
    \end{enumerate}
\end{lemma}

\begin{proof}
    \eqref{i:dir-1} holds by \cite[Proposition~VI.1.3.(ii)]{GierzHofmannEtAl2003}.
    \eqref{i:codir-1} is dual to \eqref{i:dir-1}, in the sense of \cref{r:dual}.
\end{proof}

\begin{lemma}\label{l:preserve-codir-infima}
    Every continuous order-preserving map between compact ordered spaces preserves directed suprema and codirected infima.
\end{lemma}

\begin{proof}
    We prove the preservation of directed suprema.
    Let $f \colon X \to Y$ be a continuous order-preserving map between compact ordered spaces, and let $D$ be a directed subset of $X$.
    We recall from \cref{l:exist-dirsup} that $X$ and $Y$ admit all directed suprema.
    Since $f$ is continuous, $f^{-1}[\d \sup f[D]]$ is a closed subset of $X$.
    Since $f^{-1}[\d \sup f[D]]$ contains $D$, it contains also the topological closure of $D$.
    Since $\sup D$ belongs to the topological closure of $D$ by \cref{l:exist-dirsup}, we have $\sup D \in f^{-1}[\d \sup f[D]]$, i.e.\ $f(\sup D) \leq \sup f[D]$.
    The converse inequality $\sup f[D] \leq f(\sup D)$ holds by monotonicity of $f$.
    This proves the preservation of directed suprema.

    The preservation of codirected infima is dual, in the sense of \cref{r:dual}.
\end{proof}

\begin{lemma} \label{l:equalizer}
    Let $h \colon E \to X$ be an equalizer of two morphisms $f,g \colon X \rightrightarrows Y$ in $\CompOrd$ with a common retraction, and let $K$ be a closed subset of $X$.
    \begin{enumerate}
        \item \label{i:eq-ks} 
        If $\u f[K] = \u g[K]$, then $\u K = \u(K \cap \ima h)$.
        
        \item \label{i:eq-c} 
        If $\d f[K] = \d g[K]$, then $\d K = \d(K \cap \ima h)$.
        
        \item \label{i:eq-ud}
        If $\ud f[K] = \ud g[K]$, then $\ud K = \ud (K \cap \ima h)$.
    \end{enumerate}
\end{lemma}

\begin{proof}
    \eqref{i:eq-ks}.
    Suppose $\u f[K] = \u g[K]$.
    
    \begin{claim}\label{cl:up-one-step}
        For all $x \in K$, there is $y \in K$ such that $y \leq x$ and $g(y) \leq f(x)$.
    \end{claim}
    
    \begin{claimproof}
        Let $x \in K$.
        Since $f(x) \in \u f[K] = \u g[K]$, there is $y \in K$ such that $g(y) \leq f(x)$.
        Since $f$ and $g$ have a common retraction $k \colon Y \to X$, we have
        $y = k(g(y)) \leq k(f(x)) = x$.
    \end{claimproof}
    
    Similarly, one proves the following claim.
    \begin{claim}\label{cl:up-one-step-bis}
        For all $x \in K$, there is $y \in K$ such that $y \leq x$ and $f(y) \leq g(x)$.
    \end{claim}
    
    Fix $x \in K$, and let us prove $x \in \u(K \cap \ima h)$.
    By Claims~\ref{cl:up-one-step} and \cref{cl:up-one-step-bis}, there is a chain $C$
    \[
    \dots \leq d_3' \leq d_3 \leq d_2' \leq d_2 \leq d_1' \leq d_1 \leq x
    \]
    of elements of $K$ such that 
    \begin{equation} \label{e:interchange}
        \dots \leq f(d_3') \leq g(d_3) \leq f(d_2') \leq g(d_2) \leq f(d_1') \leq g(d_1) \leq f(x).
    \end{equation}
    By \cref{l:exist-dirsup}, and since $K$ is closed, $C$ has an infimum $d$ belonging to $K$.
    Then
    \begin{align*}
        f(d) & = \inf f[C] && \text{by \cref{l:preserve-codir-infima}}\\
        & = \inf g[C] && \text{by \eqref{e:interchange}}\\
        & = g(d) && \text{by \cref{l:preserve-codir-infima}.}
    \end{align*}
    Since $f(d) = g(d)$ and $h$ is the equalizer of $f$ and $g$, by \cref{r:equalizer-compord} we have $d \in \ima h$.
    This shows $K \subseteq \u(K \cap \ima h)$, which implies $\u K \subseteq \u(K \cap \ima h)$.
    The converse inclusion is immediate.
    
    \eqref{i:eq-c}. This is the dual of \eqref{i:eq-ks}.
    
    \eqref{i:eq-ud}.
    Suppose $\ud f[K] = \ud g[K]$.
    Then $\u f[K] = \u \ud f[K] = \u \ud g[K] = \u g[K]$.
    Thus, by the part \eqref{i:eq-ks} of the present lemma, $K \subseteq \u (K \cap \ima h)$.
    Analogously, $K \subseteq \d (K \cap \ima h)$.
    Therefore, 
    \[
    K \subseteq (\u (K \cap \ima h)) \cap (\d (K \cap \ima h)) = \ud (K \cap \ima h),
    \]
    and this implies $\ud K \subseteq \ud (K \cap \ima h)$. The converse inclusion is immediate.
\end{proof}

\begin{proposition} \label{p:compord-coreflexive-equalizers}
    The functor $\Vc \colon \CompOrd \to \CompOrd$ preserves coreflexive equalizers.
\end{proposition}

\begin{proof}
    Let $h \colon E \to X$ be an equalizer of two morphisms $f,g \colon X \rightrightarrows Y$ in $\CompOrd$ with a common retraction, and let us prove that $\Vc h$ is an equalizer of $\Vc f$ and $\Vc g$.
    By \cref{l:preserve-regmono}, $\Vc h$ is a regular monomorphism.
    By functoriality of $\Vc$, we have $\Vc f \circ \Vc h = \Vc (f \circ h) = \Vc (g \circ h) = \Vc g \circ \Vc h$.	
    Let $K \in \Vc X$ be such that $(\Vc f)(K) = (\Vc g)(K)$.
    We have
    \begin{align*}
        K & = \ud K && \text{since $K$ is convex}\\
        & = \ud (K \cap \ima h) && \text{by \cref{l:equalizer}\eqref{i:eq-ud}}\\
        & = \ud h[h^{-1}[K]] && \text{since $K \cap \ima h = h[h^{-1}[K]]$}\\
        & = (\Vc h)(h^{-1}[K]).
    \end{align*}
    Therefore, $K$ belongs to the image of $\Vc h$.
    By \cref{r:equalizer-compord}, $\Vc h$ is the equalizer of $\Vc f$ and $\Vc g$.
\end{proof}

\begin{theorem} \label{t:compord-composition-is-monadic-Vietoris}
    Let $G$ be a comonadic functor from $\CompOrd$ to a category $\cat{C}$.
    The composite $\CoAlg(\Vc) \to \CompOrd \xrightarrow{G} \cat{C}$ of the forgetful functor and $G$ is comonadic. 
\end{theorem}

\begin{proof}
    The functor $\Vc$ is a covarietor by \cref{p:compord-monadic}, and it preserves coreflexive equalizers by \cref{p:compord-coreflexive-equalizers}.
    By \cref{t:composition-monadic}, the composite $\CoAlg(\Vc ) \to \CompOrd \xrightarrow{G} \cat{C}$ is comonadic.
\end{proof}

\begin{theorem} \label{t:compord-monadicity}
    The opposite $\CoAlg(\Vc )^\op$ of the category of coalgebras for the convex Vietoris functor on compact ordered spaces is monadic over $\Set$.
\end{theorem}

\begin{proof}
    By \cite{Abbadini2019} (see \cite{AbbadiniReggio2020} for a direct proof), the representable functor
    \[
    \hom_\CompOrd(-,[0,1]) \colon \CompOrd^\op \to \Set
    \]
    is monadic.
    Then, by \cref{t:compord-composition-is-monadic-Vietoris}, the composite 
    \[
    \CoAlg(\Vc )^\op \xrightarrow{U^\op} \CH^\op \xrightarrow{G} \Set
    \]
    is monadic.
\end{proof}

\begin{remark} \label{rem:ancoraBarr}
   As for \Cref{rem:cocomBarr},  from the monadicity result in \cref{t:compord-monadicity} it follows that $\CoAlg(\Vc)^\op$ is  (co)complete and Barr-exact.
\end{remark}

\begin{remark}
    From the proof of \cref{t:compord-monadicity} we can extract the description of a monadic functor $\CoAlg(\Vc )^\op \to \Set$, as follows.
    To an object $f \colon X \to \Vc (X)$ of $\CoAlg(\Vc )$ we associate the set $\hom_{\CompOrd}(X, [0,1])$.
    To a morphism $g \colon X_1 \to X_2$ in $\CoAlg(\Vc )$ from $f_1 \colon X_1 \to \Vc (X_1)$ to $f_2 \colon X_2 \to \Vc (X_2)$ we associate the function
    \begin{align*}
        \hom(X_2, [0,1]) & \longrightarrow \hom(X_1, [0,1])\\
        h & \longmapsto h \circ g.
    \end{align*}
\end{remark}


\section{Upper and lower Vietoris on stably compact spaces} \label{s:upper-lower}

In this section, we consider the upper and lower Vietoris constructions, which in domain theory are also known respectively as the Smyth and Hoare powerdomain (both due to Smyth).
The upper and lower Vietoris hyperspaces of a compact Hausdorff space in general are not compact Hausdorff and thus it is convenient to place ourselves in the larger class of stably compact spaces, which is closed under these constructions.

In \cite{HofmannNevesEtAl2018}, using their result that the opposite of the category of stably compact spaces and perfect maps is an $\aleph_1$-ary quasivariety, the authors show that the opposite of the category of coalgebras for the lower Vietoris functor on stably compact spaces and perfect maps is also an $\aleph_1$-ary quasivariety of (infinitary) algebras.
In \cite{Abbadini2019} (see \cite{AbbadiniReggio2020} for a shorter proof) the opposite of the category of stably compact spaces was proved to be in fact an $\aleph_1$-ary \emph{variety}.
In the first author's PhD thesis \cite[Conclusions]{Abbadini2021}, it was observed that a combination of these results could be used to infer that also the opposite of the category of coalgebras for the lower Vietoris functor on stably compact spaces is a(n $\aleph_1$-ary) variety.
In this section, without committing to any particular signature and axioms (but only to a choice of a dualizing object: $[0,1]$), we provide a self-contained proof (\cref{t:stcomp-monadicity}) of the latter fact.

We recall the notion of a stably compact space (see e.g.\ \cite{Lawson2011}), which, following \cite{Jung2004}, is the $T_0$ analogue of the notion of a compact Hausdorff space.
In fact, the theory of stably compact spaces is in many ways analogous to that of compact Hausdorff spaces.

\begin{definition}
    A \emph{stably compact space} is a $T_0$ topological space that is
    \begin{enumerate}
        \item coherent (i.e., the intersection of two compact saturated sets is compact saturated) and compact,
        \item locally compact (given an open $U$ and an element $x \in U$, there are an open $V$ and a compact $K$ such that $x \in V \subseteq K \subseteq U$),
        \item well-filtered (i.e., if the intersection of a filtered family of compact saturated sets is contained in an open $U$, then some member of the family is contained in $U$).
    \end{enumerate}

    A function $f \colon X \to Y$ between stably compact spaces is called \emph{perfect} if it is continuous and the preimage under $f$ of each compact saturated set in $Y$ is compact saturated \cite[Def.~9.4.1]{GoubaultLarrecq2013}. For stably compact spaces, perfect maps coincide with proper maps defined as in \cite[Def.~VI-6.20]{GierzHofmannEtAl2003}, see \cite[Lemma~VI-6.20]{GierzHofmannEtAl2003}. We let $\StComp$ denote the category of stably compact spaces and perfect maps.
\end{definition}

There is a close connection between stably compact spaces and Nachbin's compact ordered spaces.
In fact, the categories $\StComp$ and $\CompOrd$ are isomorphic (concretely over $\Set$), as first illustrated in \cite{GierzHofmannEtAl1980}.
The isomorphism $\CompOrd \to \StComp$ sends a compact ordered space $X$ to the stably compact space with the same underlying set and the topology defined by the open upsets of $X$.
Its inverse functor $\StComp \to \CompOrd$ uses the specialisation order of a topological space, defined by $x \leq y$ if and only if every open set containing $x$ contains $y$.
It maps a stably compact space $(X, \tau)$ to the space with the same set $X$ as underlying set and equipped with the patch topology (i.e.\ the topology generated by the open subsets of $X$ and by the complements of the compact saturated subsets of $X$) and the specialisation order (with respect to the original topology $\tau$). For more details we refer to \cite[Proposition~2.14]{Jung2004} and \cite[Section~VI-6]{GierzHofmannEtAl2003}.

\begin{remark}
    When we consider a partial order on a stably compact space, we assume it to be the specialisation order.
\end{remark}

We recall the lower and upper Vietoris constructions. For more details, see \cite{Schalk1993}.

\begin{definition} \label{d:upper-lower-Vietoris}
    Let $X$ be a stably compact space.
    \begin{enumerate}
        \item 
        The \emph{upper Vietoris hyperspace $\Vu X$ of $X$} is the set of compact saturated subsets of $X$ equipped with the topology generated by the sets\footnote{This is in accordance with the notation used in \cite[Definition~5.2]{Lawson2011} (with the only difference that we include also the empty set), or in \cite{GehrkeGool2022} for spectral spaces.}
        \[
        \Box U \coloneqq \{K \in \Vu X \mid K \subseteq U\} \quad  \quad \text{(}U \text{ open subset of }X\text{)}.
        \]
        The specialization order is reverse inclusion.
        
        \item
        The \emph{lower Vietoris hyperspace $\Vl X$ of $X$} is the set of closed subsets of $X$ equipped with the topology generated by the sets
        \[
        \Diamond U \coloneqq \{K \in \Vu X \mid K \cap U \neq \varnothing\} \quad \quad \text{(}U \text{ open subset of }X\text{)}.
        \]
        The specialization order is inclusion.
    \end{enumerate}
\end{definition}
\begin{definition} \hfill
	\begin{enumerate}
		\item 
		The \emph{upper Vietoris functor} is the functor $\Vu \colon \StComp \to \StComp$ that maps each stably compact space to its upper Vietoris hyperspace, and each morphism $f \colon X \to Y$ of stably compact spaces to the function
	    \begin{align*}
	        \Vu f \colon \Vu X & \longrightarrow \Vu Y\\
	        K & \longmapsto \u f[K],
	    \end{align*}
	    where $\u f[K]$ is the up-closure of $f[K]$ in the specialization order.	    
	    \item
	    The \emph{lower Vietoris functor} is the functor $\Vl \colon \StComp \to \StComp$ that maps each stably compact space to its lower Vietoris hyperspace, and each morphism $f \colon X \to Y$ of stably compact spaces to the function 
	    \begin{align*}
	        \Vl f \colon \Vl X & \longrightarrow \Vl Y\\
	        C & \longmapsto \d f[C],
	    \end{align*}
	    where $\d f[C]$ is the down-closure of $f[C]$ in the specialization order.
    \end{enumerate}
\end{definition}

In \cite[Theorem~4.2]{HofmannNevesEtAl2018} it is shown that the opposite of the category $\CoAlg(\Vl)$ of coalgebras for the Vietoris functor is equivalent to a quasivariety of infinitary algebras (in fact, an $\aleph_1$-ary quasivariety).
In \cref{t:stcomp-monadicity} below we strengthen this result and show that it is in fact an (infinitary) variety, i.e.\ $\CoAlg(\Vl)^\op$ is monadic over $\Set$.
The same result holds for $\Vu$.
To prove it, we first prove that the forgetful functors $\CoAlg(\Vu) \to \StComp$ and $\CoAlg(\Vl) \to \StComp$ are comonadic (\cref{c:comonadic}).

\begin{remark} \label{r:dual-vietoris}
    Let $\dG \colon \StComp \to \StComp$ denote the automorphism of $\StComp$ given by taking the de Groot dual\footnote{Recall that, for compact ordered spaces, the de Groot dual coincides with taking the dual order.}.
    Let $X$ be a stably compact space.
    The following diagram commutes up to natural isomorphisms (cf.~\cite[Theorem~3.1]{GoubaultLarrecq2010}).
    \[
    \begin{tikzcd}[column sep=7em, row sep=huge]
        \StComp \arrow[<->]{r}{\dG} \arrow[swap]{d}{\Vu} & \StComp \arrow{d}{\Vl} \\
        \StComp \arrow[<->, swap]{r}{\dG} & \StComp .
    \end{tikzcd}
    \]
\end{remark}

\begin{proposition}\label{p:upper-codirected-limits}
    Both the upper and the lower Vietoris functor $\Vu, \Vl \colon \StComp \to \StComp$ preserve codirected limits.
\end{proposition}

\begin{proof}
    By \cite[Corollary~3.33]{HofmannNevesEtAl2019}, $\Vl$ preserves codirected limits.
    By \cref{r:dual-vietoris}, the same holds for $\Vu$.
\end{proof}

\begin{corollary} \label{c:stcomp-monadic}
    The upper and lower Vietoris functors $\Vu, \Vl \colon \StComp \to \StComp$ are covarietors.
\end{corollary}

\begin{proof}
    By \cref{p:upper-codirected-limits}, the functors $\Vu$ and $\Vl$ preserve $\omega^\op$-limits. Thus, by \cref{p:chains,t:free-algebra-construction}, they are covarietors.
\end{proof}

\begin{corollary} \label{c:comonadic}
    The forgetful functors 
    \[
    \CoAlg(\Vu) \to \StComp \quad\text{ and }\quad \CoAlg(\Vl) \to \StComp
    \]
    are comonadic.
\end{corollary}

To prove that $\CoAlg(\Vu)^\op$ (and similarly for $\Vl$) is monadic over $\Set$, we prove that the composite of the two monadic functors 
\[
\CoAlg(\Vu)^\op \to \StComp^\op \xrightarrow{\hom(-,[0,1])} \Set
\]
is monadic.
To prove so, in light of \cref{t:composition-monadic}, it is enough to prove that $\Vu$ preserves coreflexive equalizers. We will achieve this result in \cref{p:upper-compord-codirected-limits}, but we will need some preliminary lemmas before.

\begin{lemma}[{\cite[Theorem 2.6]{HofmannNevesEtAl2018}}] \label{l:regular-mono-char}
    The regular monomorphisms in $\StComp$ are precisely the order-reflecting morphisms.
\end{lemma}

\begin{proposition} \label{p:upper-preserve-regmono}
    The upper and lower Vietoris functors $\Vu, \Vl$ preserve regular monomorphisms.
\end{proposition}

\begin{proof}
    Let $f \colon X \to Y$ be a regular monomorphism in $\StComp$, and let us prove that $\Vu f$ is order-reflecting.
    Let $K, L \in \Vu X$ with $(\Vu f) (K) \leq (\Vu f)(L)$, i.e.\ $\u f[K] \supseteq \u f[L]$.
    By \cref{l:regular-mono-char}, $f$ is order-reflecting.
    Then, by \cref{l:reg-mono}, $K \supseteq L$, i.e.\ $K \leq L$.
    Thus, $\Vu f$ is order-reflecting.
    By \cref{l:regular-mono-char}, $\Vu f$ is a regular monomorphism.
    The proof for $\Vl$ is analogous.
\end{proof}

\begin{remark} \label{r:equalizer-stcomp}
    By \cref{r:equalizer-compord}, a morphism $h \colon Z \to X$ in $\StComp$ is the equalizer of $f, g \colon X \rightrightarrows Y$ if and only if (i) an element $x \in X$ belongs to the image of $h$ if and only if $f(x) = g(x)$, and (ii) $h$ is order-reflecting.
\end{remark}

\begin{proposition} \label{p:upper-compord-codirected-limits}
    The upper and lower Vietoris functors $\Vu, \Vl \colon \StComp \to \StComp$ preserve coreflexive equalizers.
\end{proposition}

\begin{proof}
    By \cref{r:dual-vietoris}, it is enough to prove it for the upper Vietoris functor.
    Let $h \colon E \to X$ be an equalizer of two morphisms $f,g \colon X \rightrightarrows Y$ in $\StComp$ with a common retraction, and let us prove that $\Vu h$ is an equalizer of $\Vu f$ and $\Vu g$.
    By \cref{p:upper-preserve-regmono}, $\Vu h$ is a regular monomorphism.
    By functoriality of $\Vu$, we have
    \[
    \Vu f \circ \Vu h = \Vu(f \circ h) = \Vu (g \circ h) = \Vu g \circ \Vu h.
    \]
    Let $K \in \Vu X$ be such that $(\Vu f)(K) = (\Vu g)(K)$.
    We have
    \begin{align*}
        K & = \u K && \text{since $A$ is upward-closed}\\
        & = \u (K \cap \ima h) && \text{by \cref{l:equalizer}\eqref{i:eq-ks}}\\
        & = \u(h[h^{-1}[K]]) && \text{since $K \cap \ima h = h[h^{-1}[K]]$}\\
        & = (\Vu h)(h^{-1}[K]).
    \end{align*}
    Therefore, $K$ belongs to the image of $\Vu h$.
    By \cref{r:equalizer-stcomp}, $\Vu h$ is the equalizer of $\Vu f$ and $\Vu g$.
\end{proof}

\begin{theorem} \label{t:stcomp-composition-is-monadic-Vietoris}
    Let $G$ be a comonadic functor from $\StComp$ to a category $\cat{C}$.
    The composites 
    \[
    \CoAlg(\Vu) \xrightarrow{U} \StComp \xrightarrow{G} \cat{C}\quad\text{ and }\quad\CoAlg(\Vl) \xrightarrow{U} \StComp \xrightarrow{G} \cat{C}
    \]
    are comonadic.
\end{theorem}

\begin{proof}
    By \cref{r:dual-vietoris}, it is enough to prove the statement for $\Vu$.
    By \cref{c:stcomp-monadic}, $\Vu$ is a covarietor.
    By \cref{p:upper-compord-codirected-limits}, $\Vu$ preserves coreflexive equalizers.
    Then, by \cref{t:composition-monadic}, the composite below is comonadic.
    \[
    \CoAlg(\Vu) \xrightarrow{U} \StComp \xrightarrow{G} \cat{C} \qedhere
    \]
\end{proof}

\begin{notation} \label{n:01up}
    We let $[0,1]^\uparrow$ (resp.\ $[0,1]^\downarrow$) denote the stably compact space whose underlying set is the unit interval $[0,1]$ and whose topology is the set of upsets (resp.\ downsets) of $[0,1]$ that are open in the Euclidean topology.
\end{notation}

\begin{theorem} \label{t:stcomp-monadicity}
    The opposites $\CoAlg(\Vu )^\op$ and $\CoAlg(\Vl)^\op$ of the categories of coalgebras for the upper and lower Vietoris functors are monadic over $\Set$.
\end{theorem}

\begin{proof}
    By \cite{Abbadini2019} (see \cite{AbbadiniReggio2020} for a shorter proof) the representable functor
    \[
    \hom_\StComp(-,[0,1]^\uparrow) \colon \StComp^\op \to \Set
    \]
    is monadic.  Then, by \cref{t:stcomp-composition-is-monadic-Vietoris}, the composite below is monadic.
    \[
    \CoAlg(\Vu )^\op \xrightarrow{U^\op} \CH^\op \xrightarrow{G} \Set \qedhere
    \]
\end{proof}

\begin{remark}
   Similarly to \Cref{rem:ancoraBarr} (and \Cref{rem:cocomBarr}), from the monadicity result in \cref{t:stcomp-monadicity} we obtain that the categories $\CoAlg(\Vu)^\op$ and $\CoAlg(\Vl)^\op$, are (co)complete and Barr-exact.
\end{remark}

\begin{remark}
    We strengthened the results in \cite{HofmannNevesEtAl2018} stating that the opposite of $\CoAlg(\Vl)$ is equivalent to an $\aleph_1$-ary quasivariety: this $\aleph_1$-ary quasivariety is a variety, and, in fact, an $\aleph_1$-ary \emph{variety}.
    Indeed, any quasivariety that is equivalent to a variety is itself a variety (since the only missing property is the effectiveness of equivalence relations, which is a categorical property that does not depend on the axiomatization).
\end{remark}

\begin{remark}
    The proof of \cref{t:stcomp-monadicity} provides an example of a monadic functor from $\CoAlg(\Vu )^\op$ to $\Set$, as follows.
    To an object $f \colon X \to \Vu X$ of $\CoAlg(\Vu )$ we associate the set $\hom_{\StComp}(X, [0,1]^\uparrow)$.
    To a morphism $g \colon X_1 \to X_2$ in $\CoAlg(\Vu )$ from $f_1 \colon X_1 \to \Vu X_1$ to $f_2 \colon X_2 \to \Vu X_2$ we associate the function
    \begin{align*}
        \hom_{\StComp}(X_2, [0,1]^\uparrow) & \longrightarrow \hom_{\StComp}(X_1, [0,1]^\uparrow)\\
        h & \longmapsto h \circ g.
    \end{align*}
    Analogous considerations hold for $\Vl$.
\end{remark}

\section{Generating the algebraic theories}

The general theory initiated by Lawvere and Linton offers a \textit{standard} procedure to axiomatize monadic categories (see \cite{manes2012algebraic} or \cite[Sec. 3.6]{manes2003monads}). In this section, 
\begin{itemize}
    \item
    we obtain an equational axiomatization of $\CoAlg(\V)^\op$ by adding to an axiomatization of $\CH^\op$ the unary operator $\Box$ (or, equivalently, $\Diamond$) and appropriate axioms;
    
    \item 
    we obtain an equational axiomatization of $\CoAlg(\Vc)^\op$ by adding to an axiomatization of $\CompOrd^\op$ the unary operators $\Box$ and $\Diamond$ and appropriate axioms;
    
    \item 
    we obtain an equational axiomatization of $\CoAlg(\Vu)^\op$ by adding to an axiomatization of $\StComp^\op$ the unary operator $\Box$ and appropriate axioms;
    
    \item
    we obtain an equational axiomatization of $\CoAlg(\Vl)^\op$ by adding to an axiomatization of $\StComp^\op$ the unary operator $\Diamond$ and appropriate axioms.
\end{itemize}

The algebras in a variety $\mathcal{W}$ can be viewed as algebras for a monad $F \colon \Set \to \Set$ that maps a set $X$ to the underlying set of the free $\mathcal{W}$-algebra over $X$.
Then, an algebra $A$ in $\mathcal{W}$ is given by a function $F(A) \to A$ which satisfies certain properties of compatibility with the monad identity and multiplication.

In special cases, this description gets simplified in the sense that the variety $\mathcal{W}$ is isomorphic to the category of algebras for an endofunctor $T \colon \Set \to \Set$ (which generally differs from the endofunctor on $\Set$ describing the monad); no monad units and multiplications are needed in these cases.

\begin{example}[{\cite[Example, p.~114, Section~III.3]{AdamekTrnkova1990}}] \label{ex:commutative}
    Consider the variety $\mathcal{Z}$ of commutative magmas, whose language consists of a single binary operation $+$ and whose single axiom is commutativity.
    For every $X$, let $TX$ be the set of subsets of $X$ of cardinality $1$ or $2$.
    Then, an algebra in $\mathcal{Z}$ can be identified with a function from $TX$ to $X$, as follows.
    To an algebra $A$ in $\mathcal{Z}$ one associates the function $TA \to A$ that maps a set $\{x,y\}$ (with $x$ and $y$ possibly coinciding) to $x + y$.
    Note that commutativity ensures that this function is well-defined.
    The assignment $T$ can be made into an endofunctor on $\Set$ by mapping a function $f \colon X \to Y$ to the function that maps a set $\{x,y\}$ (with $x$ and $y$ possibly coinciding) to its image $f[\{x,y\}]$ under $f$.
    A homomorphism $f \colon A \to B$ between algebras in $\mathcal{Z}$ corresponds to a function $g \colon A \to B$ making the following diagram commuting.
    \[
    \begin{tikzcd}
        TA \arrow{r}{Tg} \arrow{d}{} & TB \arrow{d}\\
        A \arrow{r}{g} & B
    \end{tikzcd}
    \]
    
    Notice that commutativity 
    \[
    x + y = y + x
    \]
    is an equation of pure rank 1, i.e.\ every variable is under exactly one operation symbol.
    
    The set $TX$ can then be identified with the (classes of equivalence of) terms of rank $1$ in variables from $X$ (i.e.\ the terms of type $x + y$ for $x,y \in X$).
    Notice that $T$ is different from the functor $T^* \colon \Set \to \Set$ that is part of the monad describing the variety $\mathcal{Z}$.
    The functor $T^*$ is the algebraically-free monad, which maps a set $X$ to the set of equivalence classes of terms in variables from $X$; $T^*X$ contains elements of the form $(x+y) + z$ (of rank $2$) or variables $x$ (of rank $0$), which are not present in $TX$.
\end{example}

In fact, if a variety $\mathcal{Z}$ is axiomatized by pure rank $1$ axioms, then $\mathcal{Z}$ is isomorphic to the category of algebras for an endofunctor on $\Set$ (namely, the functor that sends $X$ to the set of equivalence classes of terms in variables from $X$ of pure rank $1$), cf.\ \cite[Corollary, p.~141, Section~III.4.9]{AdamekTrnkova1990}, \cite[Section~5.A]{BalanKurzEtAl2015}.
More generally, if a variety $\mathcal{Z}$ is obtained from a variety $\mathcal{W}$ (in \cref{ex:commutative}, $\mathcal{W} = \Set$) by adding operations axiomatized by pure rank $1$ axioms, then we can view $\mathcal{Z}$-algebras as algebras for an endofunctor on $\mathcal{W}$.

\begin{example}
    The variety $\MA$ of modal algebras is obtained from the variety $\BA$ of Boolean algebras by adding a unary function symbol $\Box$ axiomatized by pure rank $1$ equations, namely
    \[
        \Box(x \land y) = \Box x \land \Box y \quad \text{ and }\quad \Box 1 = 1.
    \]
    The category $\MA$ is isomorphic to the category of algebras for an endofunctor on $\BA$ that maps a Boolean algebra to the free Boolean algebra over its underlying meet-semilattice \cite[Proposition~3.12]{KupkeKurzEtAl2004}.
\end{example}

\begin{example}
    The variety of positive modal algebras is obtained from the variety of bounded distributive lattices by adding unary function symbols $\Box$ and $\Diamond$ axiomatized by pure rank $1$ equations, namely
    \begin{enumerate}
        \item 
        $\Box(x \land y) = \Box x \land \Box y$,
        
        \item
        $\Box 1 = 1$,
        
        \item
        $\Diamond (x \lor y) = \Diamond x \lor \Diamond y$,
        
        \item
        $\Diamond 0 = 0$,
        
        \item
        $\Box x \land \Diamond y = \Box x \land \Diamond (x \land y)$,
        
        \item
        $\Diamond x \lor \Box y = \Diamond x \lor \Box(x \lor y)$.
    \end{enumerate}
    In fact, the variety of positive modal algebras can be seen as the category of algebras for an endofunctor on the category of bounded distributive lattices and homomorphisms.
\end{example}

\begin{remark} \label{r:presentation-rank}
    Suppose $G \colon \mathcal{W} \to \Set$ is a monadic functor, with left adjoint $F$, and $(GF, \eta, \mu)$ the corresponding monad.
    For a function $f \colon X \to GF(Y)$, let $\overline{f} \colon F(X) \to F(Y)$ denote the unique morphism such that $\overline{f} \circ \eta_X = f$ (whose existence is guaranteed by the fact that $\eta$ is a unit); in other words, $\overline{f} = \mu_T \circ GF(f)$; in more algebraic terms, $\overline{f}$ is the extension of $f$ from the set $X$ of free generators to the free algebra $F(X)$.
    Suppose $T \colon \mathcal{W} \to \mathcal{W}$ is an endofunctor such that the forgetful functor $\Alg(T) \to \mathcal{W}$ is monadic.
    Then the algebraic theory of $\Alg(T)$ is determined by the following information.

    Terms:
    \begin{enumerate}
        \item
        The set of (equivalence classes of) terms of rank $0$ of arity $X$ is $GF(X)$.
        
        \item 
        The set of (equivalence classes of) terms of rank $1$ of arity $X$ is $GTF(X)$.
    \end{enumerate}
    
    Composition of terms (denoted with $*$ to distinguish it from the composition of functions, denoted with $\circ$):
    \begin{enumerate}
        \item 
        The composite $\tau * \rho$ of a term $\tau \in GF(X)$ of rank $0$ and arity $X$ and an $X$-tuple $\rho \colon X \to GF(Y)$ of terms of rank $0$ and arity $Y$ is $G(\overline{\rho}) (\tau) \in GF(Y)$.
        
        \item 
        Composition of terms of rank $0$ and rank $1$ is as follows.
        \begin{enumerate}
            \item 
            The term-composition $\tau * \rho$ of a term $\tau \in GF(X)$ of rank $0$ and arity $X$ and an $X$-tuple $\rho \colon X \to GTF(Y)$ of terms of rank $1$ and arity $Y$ is $G(\overline{\rho})(\tau) \in GTF(Y)$.
            
            \item 
            The term-composition $\tau * \rho$ of a term $\tau \in GTF(X)$ of rank $1$ and arity $X$ and an $X$-tuple $\rho \colon X \to GF(Y)$ of terms of rank $0$ and arity $Y$ is $GT(\overline{\rho})(\tau) \in GTF(Y)$.
        \end{enumerate}
    \end{enumerate}
\end{remark}

\subsection{Vietoris on compact Hausdorff spaces}
The monadic functor 
\[
\CoAlg(\V)^\op \to \Set
\]
given by the proof of \cref{t:stcomp-monadicity} maps an object $f \colon X \to \V (X)$ to the set $\hom_{\CH}(X, [0,1])$, and maps a morphism $g \colon X_1 \to X_2$ from $f_1 \colon X_1 \to \V (X_1)$ to $f_2 \colon X_2 \to \V (X_2)$ to the precomposition by $g$:
\begin{align*}
    \hom_\CH(X_2, [0,1]) & \longrightarrow \hom_\CH(X_1, [0,1])\\
    h & \longmapsto h \circ g.
\end{align*}
In what follows, we give a (quite abstract) description of the algebraic theory of the variety associated to this monadic functor.

\begin{notation}
    Given a family $(f_i \colon X \to Y_i)_{i \in I}$ of functions, we let $\langle f_i \rangle_{i \in I}$ denote the obvious function $ X \to \prod_{i \in I}Y_i$.
\end{notation}

\begin{remark} \label{r:a-description}
    The left adjoint to Duskin's monadic functor $\hom(-, [0,1]) \colon \CH^\op \to \Set$ is the functor that maps a set $X$ to the space $[0,1]^X$ and a function $f \colon X \to Y$ to the precomposition map ${-} \circ f \colon [0,1]^Y \to [0,1]^X$.
    Then, by \cref{r:presentation-rank}, the algebraic theory of $\CoAlg(\V)^\op$ is determined by the following information.
    
    Terms:
    \begin{enumerate}
        \item 
        The (equivalence classes of) terms of rank $0$ and arity $X$ are the continuous maps from $[0,1]^X$ to $[0,1]$. (Since every continuous map from a power of $[0,1]$ to $[0,1]$ depends on at most countably many coordinates \cite[Theorem~1]{Mibu1944}, it is enough to take $X$ countable.)
        
        \item 
        The (equivalence classes of) terms of rank $1$ and arity $X$ are the continuous maps from $\V([0,1]^X)$ to $[0,1]$.
     \end{enumerate}
    
    Composition of terms:
    \begin{enumerate}
        \item 
        Composition of terms of rank $0$ is composition of functions.
        
        \item 
        Composition of terms of rank $0$ and rank $1$ is as follows.
        \begin{enumerate}
            \item 
            The term-composition $t * (\gamma_i)_{i \in I}$ of a continuous map $t \colon [0,1]^I \to [0,1]$ (so, a term of rank $0$ and arity $I$) and a family $(\gamma_i \colon \V([0,1]^X) \to [0,1])_{i \in I}$ of continuous maps  (i.e.\ of terms of rank $1$ and arity $X$) is the composite
            \[
            \V([0,1]^X) \xrightarrow{\langle\gamma_i\rangle_{i \in I}}[0,1]^I \xrightarrow{t} [0,1].
            \]
            \item
            The term-composition $\gamma * (t_x)_{x \in X}$ of a continuous map $\gamma \colon \V([0,1]^X) \to [0,1]$ (i.e.\ a term of rank $1$ and arity $X$) and a family $t = (t_x \colon [0,1]^I \to [0,1])_{x \in X}$ of continuous maps  (i.e.\ of terms of rank $0$ and arity $I$) is the composite
            \[
            \V([0,1]^I) \xrightarrow{\V(\langle t_x \rangle_{x \in X})} \V([0,1]^X)\xrightarrow{\gamma} [0,1].
            \]
        \end{enumerate}
    \end{enumerate}
\end{remark}

\begin{remark} \label{r:axiom-first-version}
    In view of \cref{r:a-description}, $\CoAlg(\V)^\op$ has the following equational axiomatization.
    
    Language:
    \begin{enumerate}
        \item 
        For every (at most countable) set $X$, we take each continuous map from $[0,1]^X$ to $[0,1]$ as a function symbol of arity $X$.
        \item
        For every set $X$, we take each continuous map from $\V([0,1]^X)$ to $[0,1]$ as a function symbol of arity $X$.
    \end{enumerate}
    
    Axioms:
    \begin{enumerate}
        \item 
        For every continuous map $t \colon [0,1]^I \to [0,1]$ and family $(\gamma_i \colon \V([0,1]^X) \to [0,1])_{i \in I}$ of continuous maps, we take the axiom
        \[
        t * (\gamma_i)_{i \in I} = t \circ \langle \gamma_i \rangle_{i \in I}.
        \]
        \item
        For every continuous map $t \colon [0,1]^I \to [0,1]$ and family $(\gamma_i \colon \V([0,1]^X) \to [0,1])_{i \in I}$ of continuous maps 
        we take the axiom
        \[
            t * (\gamma_i)_{i \in I} = t \circ \langle \gamma_i \rangle_{i \in I}.
        \]
        \item
        For every continuous map $\gamma \colon \V([0,1]^X) \to [0,1]$ and family $(t_x \colon [0,1]^I \to [0,1])_{x \in X}$ of continuous maps, we take the axiom
        \[
            \gamma * (t_x)_{x \in X} = \gamma \circ \V(\langle t_x \rangle_{x \in X}).
        \]
    \end{enumerate}
\end{remark}

In this subsection we show that, instead of taking all terms of pure rank $1$, it is sufficient to take a single unary term $\Box$ of rank $1$ (together with all the terms of rank $0$) to generate all terms.
(Since the terms of rank $0$ are generated by finitely many terms \cite{Isbell1982}, it follows that the theory of $\CoAlg(\V)$ is generated by finitely many terms.)

\begin{notation}
    In this subsection, $\Box$ denotes the function
    \begin{align*}
        \V([0,1]) & \longrightarrow [0,1]\\
        K & \longmapsto \inf K,
    \end{align*}
    while $\Diamond$ denotes the function
    \begin{align*}
        \V([0,1]) & \longrightarrow [0,1]\\
        K & \longmapsto \sup K.
    \end{align*}
    (Infima and suprema are computed in $[0,1]$, so $\Box \varnothing = 1$ and $\Diamond \varnothing = 0$.)
    (In the next subsections the functions $\Box$ and $\Diamond$ will be defined in the same way but with a different domain.)
\end{notation}

\begin{lemma} \label{l:Box-Diam-cont}
    The functions $\Box, \Diamond \colon \V([0,1]) \to [0,1]$ are continuous.
\end{lemma}

\begin{proof}
    For every $\lambda \in [0,1]$ we have
    \begin{align*}
        \Box^{-1}[(\lambda,1]] & = \{K \in \V([0,1]) \mid \inf K \in (\lambda, 1]\} = \{K \in \V([0,1]) \mid K \subseteq (\lambda, 1]\},\\
        \Box^{-1}[[0, \lambda)] & = \{K \in \V([0,1]) \mid \inf K \in [0,\lambda)\} = \{K \in \V([0,1]) \mid K \cap (0, \lambda] \neq \varnothing\},\\
        \Diamond^{-1}[(\lambda,1]] & = \{K \in \V([0,1]) \mid \sup K \in (\lambda, 1]\} = \{K \in \V([0,1]) \mid K \cap (\lambda,1] \neq \varnothing\},\\
        \Diamond^{-1}[[0,\lambda)] & = \{K \in \V([0,1]) \mid \sup K \in [0,\lambda)\} = \{K \in \V([0,1]) \mid K \subseteq [0,\lambda)\}.
    \end{align*}
    All these sets are open in the Vietoris topology.
    Thus, $\Box$ and $\Diamond$ are continuous.
\end{proof}

In classical modal logic, $\Box$ and $\Diamond$ are interdefinable: $\Diamond x = \lnot \Box \lnot x$ and $\Box x = \lnot \Diamond \lnot x$.
The next lemma shows that the same is true in the context of compact Hausdorff spaces.
Set
\begin{align*}
    \lnot \colon [0,1] & \longrightarrow [0,1]\\
    x & \longmapsto 1 - x.
\end{align*}

\begin{lemma} \label{l:interdefinable}
    \begin{enumerate}
        \item \label{i:D-from-B}
        $\Diamond = \lnot \circ \Box \circ \V(\lnot)$, i.e.\ $\Diamond = \lnot * \Box * \lnot$.

        \item \label{i:B-from-D}
        $\Box = \lnot \circ \Diamond \circ \V(\lnot)$, i.e.\ $\Box = \lnot * \Diamond * \lnot$.
    \end{enumerate}
\end{lemma}

\begin{proof}
    \eqref{i:D-from-B}.
    Let $K \in \V([0,1])$.
    Then
    \begin{align*}
        (\lnot \circ \Box \circ \V(\lnot)) (K) & = \lnot (\Box (\V(\lnot)(X)))\\
        & = \lnot (\Box (\lnot[K]))\\
        & = \lnot (\Box \{1 - x \mid x \in K\})\\
        & = 1 - \inf \{1 - x \mid x \in K\}\\
        & = \sup \{x \mid x \in K\}\\
        & = \Diamond K.
    \end{align*}

    \eqref{i:B-from-D} is analogous.
\end{proof}

\noindent In \cref{l:interdefinable}, the function $\lnot \colon [0,1] \to [0,1]$ could be replaced by any order-reversing homomorphism from $[0,1]$ to $[0,1]$.

In the next proposition we recall a fact that combines a categorical characterization of quasi-varieties (see \cite[Theorem~3.6]{Adamek2004}) and the theory of natural dualities, for an overview
of which we refer to \cite{PorstTholen1991}.
We will apply it with $\cat{C} = \CH$, $X = [0,1]$, and $U$ the underlying set functor.

\begin{proposition}[{See e.g.\ \cite[Proposition~2.8]{Abbadini2021}}]\label{p:dual}
    Let $X$ be a regular injective regular cogenerator of a complete category $\cat{C}$ and let $U \colon \cat{C} \to \Set$ be a faithful representable functor.
    Let $\Sigma$ be the signature whose elements of arity $\kappa$ (for each set $\kappa$) are the morphisms from $X^\kappa$ to $X$, and let $\overline{X}$ be the $\Sigma$-algebra whose underlying set is $U(X)$ and on which the interpretation of any operation symbol $f \colon X^\kappa \to X$ is $U(f) \colon U(X)^\kappa \to U(X)$.
    Then, $\cat{C}$ is dually equivalent to $\SP(\overline{X})$.
\end{proposition}

\begin{remark} \label{r:explicit-duality}
    The contravariant functor from $\cat{C}$ to $\SP(\overline{X})$ in the duality in \cref{p:dual} maps an object $Y$ to the $\Sigma$-algebra $\hom_{\cat{CH}}(Y,X)$ equipped with pointwise defined operations, and maps a morphism $f \colon Y_1 \to Y_2$ to the precomposition by $f$.
\end{remark}

The following can be thought of as a version of the Stone-Weierstrass theorem \cite{Weierstrass1885,Stone1948}. The theorem is a special case of \cite[Proposition 3.6]{hofmann2002generalization}, but we provide a proof to keep the paper self-contained.

\begin{proposition} \label{p:it-is-all-the-functions}
    Let $\mathcal{C}$ be a set of continuous maps from a compact Hausdorff space $X$ to $[0,1]$. Suppose that $\mathcal{C}$ is point-separating and closed under pointwise application of any continuous map from a power of $[0,1]$ to $[0,1]$.
    Then $\mathcal{C}$ is the set of continuous maps from $X$ to $[0,1]$.
\end{proposition}

\begin{proof}
    We let $\iota \colon \mathcal{C} \to \hom_\CH(X, [0,1])$ denote the inclusion function.
    By hypothesis, $\mathcal{C} \in \SP([0,1])$.
    By \cref{p:dual,r:explicit-duality}, there is $Y \in \CH$ and an isomorphism $\gamma \colon \hom_\CH(Y, [0,1]) \to \mathcal{C}$ in $\SP([0,1])$.
    Moreover, there is a continuous map $f \colon X \to Y$ such that the map $\iota \circ \gamma \colon \hom_\CH(Y, [0,1]) \to \hom_\CH(X, [0,1])$ is the precomposition ${-}\circ f$ by $f$.
    The function $\iota \circ \gamma$ is injective and thus a monomorphism in $\SP([0,1])$; therefore, $f$ is an epimorphism in $\CH$ and hence surjective.
    To prove that $f$ is injective, let $x,y \in X$ with $f(x) = f(y)$.
    Then, for every $g \in \hom_\CH(Y, [0,1])$ we have
    \[
    (\iota \circ \gamma)(g)(x) = (g \circ f)(x) = g(f(x)) = g(f(y)) = (g \circ f)(y) = (\iota \circ \gamma)(g)(y).
    \]
    Since the image of $\iota \circ \gamma$ is separating, it follows that $x = y$.
    Hence, $f$ is a continuous bijection, i.e.\ an isomorphism.
    Thus, ${-} \circ f = \iota \circ \gamma$ is an isomorphism, and thus $\iota$ is surjective.
\end{proof}

In view of \cref{r:a-description}, the following says that the terms of $\CoAlg(\V)^\op$ are generated by the terms of $\CH^\op$ and $\Box$.

\begin{theorem} \label{t:generated-by-Box}
    Let $X$ be a set.
    For every continuous map $f \colon \V([0,1]^X) \to [0,1]$ there are
    a(n at most countable) set $Y$, a continuous map $s \colon [0,1]^Y \to [0,1]$ and a family $(h_y \colon [0,1]^X \to [0,1])_{y \in Y}$ of continuous maps such that
    \[
    f = s \circ \langle\Box \circ \V h_y\rangle_{y \in Y}.
    \]
\end{theorem}

\begin{proof}
    Let $\mathcal{C}$ be the set of functions $f \colon \V([0,1]^X) \to [0,1]$ for which there are a set $Y$, a continuous map $s \colon [0,1]^Y \to [0,1]$ and a family $(h_y \colon [0,1]^X \to [0,1])_{y \in Y}$ of continuous maps such that $f = s \circ \langle\Box \circ \V h_y\rangle_{y \in Y}$.
    We prove that $\mathcal{C}$ is the set of continuous maps from $\V([0,1]^X)$ to $[0,1]$.
    By \cref{l:Box-Diam-cont}, $\Box$ is continuous; hence, every function in $\mathcal{C}$ is continuous.
    Since $\mathcal{C}$ is closed under every continuous map from some power of $[0,1]$ to $[0,1]$, by \cref{p:it-is-all-the-functions} it is enough to prove that $\mathcal{C}$ is point-separating, i.e.\ that for all distinct $K, L \in \V([0,1]^X)$ there is $f \in \mathcal{C}$ such that $f(K) \neq f(L)$.
    Let $K, L \in \V([0,1]^X)$ be distinct.
    Without loss of generality, we may suppose $K \nsubseteq L$.
    Then, there is $y \in K \setminus L$.
    By Urysohn's lemma, there is a continuous map $h \colon [0,1]^X \to [0,1]$ such that $h(y) = 0$ and $h[L] \subseteq \{1\}$.
    Then, $\Box \circ \V h \in \mathcal{C}$ and
    \begin{align*}
        (\Box \circ \V h) (K) & = \Box (\V h (K)) = \Box (h[K]) = \inf h[K] = 0,\\
        (\Box \circ \V h) (L) & = \Box (\V h (L)) = \Box (h[L]) = \inf h[L] = 1.
    \end{align*}
    Thus, $\mathcal{C}$ separates the points.
    
    One can take $Y$ to be at most countable because every continuous map from a power of $[0,1]$ to $[0,1]$ depends on at most countably many coordinates \cite[Theorem~1]{Mibu1944}.
\end{proof}

\begin{remark} \label{r:axiomatization-with-Box}
    In light of \cref{t:generated-by-Box}, an axiomatization of $\CoAlg(\V)^\op$ improving the one in \cref{r:axiom-first-version} is as follows.

    Language:
    For each $\kappa \in \omega \cup \{\omega\}$, we take each continuous function from $[0,1]^\kappa$ to $[0,1]$ as a function symbol of arity $\kappa$.
    Moreover, we take $\Box$ as a unary function symbol.
    
    Axioms:
    \begin{enumerate}
        \item 
        For all $\kappa, \lambda \in \omega \cup \{\omega\}$, each continuous map $t \colon [0,1]^\kappa \to [0,1]$ and each family $(\gamma_i \colon [0,1]^\lambda \to [0,1])_{i \in \kappa}$ of continuous maps, we take the axiom
        \[
        t * (\gamma_i)_{i \in \kappa} = t \circ \langle \gamma_i \rangle_{i \in \kappa}.
        \]
        Here, the left-hand side is a formal composition of function symbols, while the right-hand side is just one function symbol (of arity $\lambda$).

        \item
        For all $\kappa, \lambda, \kappa', \lambda' \in \omega \cup \{\omega\}$, all continuous maps $t \colon [0,1]^\kappa \to [0,1]$ and $t' \colon [0,1]^{\kappa'} \to [0,1]$, and all families $(s_{i} \colon [0,1]^{\lambda} \to [0,1])_{i \in I}$ and $(s'_{i} \colon [0,1]^{\lambda'} \to [0,1])_{i \in I'}$ of continuous maps such that
        $t \circ \langle\Box \circ \V s_i\rangle_{i \in \kappa} = t' \circ \langle\Box \circ \V(s'_i)\rangle_{i \in \kappa'}$, we take the axiom
        \[
            t * (\Box * s_i)_{i \in \kappa} = t' * (\Box * s'_i)_{i \in \kappa'}
        \]
    \end{enumerate}
\end{remark}

\begin{remark}
    Since the terms of rank $0$ are generated by finitely many terms \cite{Isbell1982}, it follows from the results above that the theory of $\CoAlg(\V)^\op$ is generated by finitely many terms. (We simply add $\Box$ to the finitely many terms of $\CH^\op$.)
    Moreover, we note that each term depends on countably many coordinates only.
\end{remark}

\begin{remark}
    The full, faithful and essentially surjective contravariant functor from $\CoAlg(\V)$ to the variety described in \cref{r:axiomatization-with-Box} is as follows.
    To a coalgebra $f \colon X \to \V(X)$ we associate the algebra with $\hom_{\CH}(X, [0,1])$ as the underlying set and with the interpretation of the operations as follows:
    \begin{enumerate}
        \item Every continuous function $[0,1]^I \to [0,1]$ is interpreted pointwise.
        
        \item The symbol $\Box$ is interpreted as follows: for $g \in \hom_{\CH}(X, [0,1])$,
        \begin{align*}
            \Box g \colon X & \longrightarrow [0,1]\\
            x & \longrightarrow \inf g[f(x)].
        \end{align*}
    \end{enumerate}
    To a morphism $g \colon X_1 \to X_2$ from $f_1 \colon X_1 \to \V(X_1)$ to $f_2 \colon X_2 \to \V(X_2)$ we associate the precomposition by $g$.
\end{remark}

\subsection{Convex Vietoris on compact ordered spaces}

We now do a similar study for the convex Vietoris functors on compact ordered spaces.

\begin{notation}
    In this subsection, $\Box$ and $\Diamond$ denote the functions from $\Vc ([0,1])$ to $[0,1]$ that map $A$ to $\inf A$ and $\sup A$, respectively.
\end{notation}

\begin{lemma} \label{l:BD-cont-order-p}
    The functions $\Box, \Diamond \colon \Vc([0,1]) \to [0,1]$ are continuous and order-preserving.
\end{lemma}

\begin{proof}
    The proof of continuity is similar to the proof of \cref{l:Box-Diam-cont}.
    
    We prove that $\Box \colon \Vc([0,1]) \to [0,1]$ is order-preserving.
    Let $K, L \in \Vc([0,1])$ with $K \leq_{\EM} L$.
    Then $\u L \subseteq \u K$, and thus $\Box K = \Box \u K = \inf \u K \leq \inf \u L = \Box \u L = \Box L$.
    Analogously, $\Diamond \colon \Vc([0,1]) \to [0,1]$ is order-preserving.
\end{proof}

A family of order-preserving functions $(f_i \colon X \to Y_i)_{i \in I}$ between posets is said to be \emph{order-separating} if $x_1 \nleq x_2$ implies that there is $i \in I$ such that $f(x_1) \nleq f(x_2)$.

Similarly to \Cref{p:it-is-all-the-functions}, the proposition below is a special case of \cite[Proposition 3.6]{hofmann2002generalization}.

\begin{proposition} \label{p:it-is-all-the-functions-order}
    Let $\mathcal{C}$ be a set of continuous order-preserving maps from a compact ordered space $X$ to $[0,1]$. Suppose that $\mathcal{C}$ is order-separating and closed under pointwise application of any continuous order-preserving map from a power of $[0,1]$ to $[0,1]$.
    Then $\mathcal{C}$ is the set of continuous order-preserving maps from $X$ to $[0,1]$.
\end{proposition}

\begin{proof}
    We let $\iota \colon \mathcal{C} \to \hom_\CompOrd(X, [0,1])$ denote the inclusion function.
    By hypothesis, $\mathcal{C} \in \SP([0,1])$.
    By \cref{p:dual,r:explicit-duality}, there is a compact ordered space $Y$ and an isomorphism $\gamma \colon \hom_\CompOrd(Y, [0,1]) \to \mathcal{C}$.
    Moreover, there is a continuous order-preserving map $f \colon X \to Y$ such that the map $\iota \circ \gamma \colon \hom_\CompOrd(Y, [0,1]) \to \hom_\CompOrd(X, [0,1])$ is the precomposition ${-}\circ f$ by $f$.
    The function $\iota \circ \gamma$ is injective and thus a monomorphism in $\SP([0,1])$; therefore, $f$ is an epimorphism in $\CompOrd$ and hence surjective.
    We prove that $f$ is order-reflecting.
    Let $x,y \in X$ and suppose $f(x) \leq f(y)$.
    Then, for every $g \in \hom_\CompOrd(Y, [0,1])$,
    \[
    (\iota \circ \gamma)(g)(x) = (g \circ f)(x) = g(f(x)) \leq g(f(y)) = (g \circ f)(y) = (\iota \circ \gamma)(g)(y).
    \]
    Since the image of $\iota \circ \gamma$ is order-separating, $x \leq y$.
    Hence, $f$ is a continuous order-reflecting bijection, i.e.\ an isomorphism.
    Thus, ${-} \circ f = \iota \circ \gamma$ is an isomorphism, whence $\iota$ is surjective.
\end{proof}

The following says that the terms of $\CoAlg(\Vc)^\op$ are generated by the terms of $\CompOrd^\op$, $\Box$ and $\Diamond$.

\begin{theorem}
    Let $X$ be a set.
    For every continuous order-preserving map $f \colon \Vc([0,1]^X) \to [0,1]$ there is a(n at most countable) set $Y$, a continuous order-preserving map $s \colon [0,1]^Y \to [0,1]$, a family $(h_y \colon [0,1]^X \to [0,1])_{y \in Y}$ of continuous order-preserving maps, and a family $(\otimes_y)_{y \in Y}$ of elements of $\{\Box, \Diamond\}$ such that
    \[
    f = s \circ \langle\otimes_y \circ \Vc h_y\rangle_{y \in Y}.
    \]
\end{theorem}

\begin{proof}
    Let $\mathcal{C}$ be the set of ordered-preserving continuous maps $f \colon \Vc([0,1]^X) \to [0,1]$ for which there are a set $Y$, a continuous order-preserving map $s \colon [0,1]^Y \to [0,1]$, a family $(h_y \colon [0,1]^X \to [0,1])_{y \in Y}$ of continuous order-preserving maps and a family $(\otimes_y)_{y \in Y}$ of elements of $\{\Box, \Diamond\}$ such that $f = s \circ \langle\otimes_y \circ \Vc h_y\rangle_{y \in Y}$.
    We shall prove that $\mathcal{C}$ is the set of continuous order-preserving maps from $\Vc([0,1]^X)$ to $[0,1]$.
    By \cref{l:BD-cont-order-p}, $\Box$ and $\Diamond$ are order-preserving and continuous, and thus every function in $\mathcal{C}$ is order-preserving and continuous.
    Since $\mathcal{C}$ is closed under every continuous map from some power of $[0,1]$ to $[0,1]$, by \cref{p:it-is-all-the-functions-order} it is enough to prove that $\mathcal{C}$ is order-separating, i.e.\ that for all $K, L \in \Vc([0,1]^X)$ with $K \not\leq_{\EM} L$ there is $f \in \mathcal{C}$ such that $f(K) \not\leq f(L)$.
    Let $K, L \in \Vc([0,1]^X)$ with $K \not\leq_{\EM} L$.
    Then, either $\u L \nsubseteq \u K$ or $\d K \nsubseteq \d L$.
	
    Case $\u L \nsubseteq \u K$.
    Then there is $y \in \u L \setminus \u K$.
    By the ordered Urysohn's lemma \cite[Theorem~1, p.~30]{Nachbin1965} (which applies to compact ordered spaces in light of \cite[Corollary of Theorem~4, p.~48]{Nachbin1965}), there is a continuous order-preserving map $h \colon [0,1]^X \to [0,1]$ such that $h(y) = 0$ and $h[K] \subseteq \{1\}$.
    Then,
    \begin{align*}
        (\Box \circ \Vc h) (K) & = \Box (\Vc h (K)) = \Box (\ud h[K]) = \inf \ud h[K] = \inf h[K] = 1,\\
        (\Box \circ \Vc h) (L) & = \Box (\Vc h (L)) = \Box \ud h[L] = \inf \Box \ud h[L] = \inf h[L] = 0.
    \end{align*}
    Then, $\Box \circ \Vc(h) \in \mathcal{C}$ and $(\Box \circ \Vc)(K) \not\leq (\Box \circ \Vc)(L)$.
	
    Case $\d K \nsubseteq \d L$.
    Then there is $y \in \d K \setminus \d L$.
    By the ordered Urysohn's lemma, there is a continuous order-preserving map $h \colon [0,1]^X \to [0,1]$ such that $h(y) = 1$ and $h[L] \subseteq \{0\}$.
    Then,
    \begin{align*}
	(\Diamond \circ \Vc h) (K) & = \Diamond (\Vc h (K)) = \Diamond (\ud h[K]) = \sup \ud h[K] = \sup h[K] = 1,\\
	(\Diamond \circ \Vc h) (L) & = \Diamond (\Vc h (L)) = \Diamond (\ud h[L]) = \sup\ud h[L] = \sup h[L] = 0.
    \end{align*}
    Then, $\Diamond \circ \Vc h \in \mathcal{C}$ and $(\Diamond \circ \Vc h)(K) \not\leq (\Diamond \circ \Vc h)(L)$.
	
    Since every continuous map $f \colon [0,1]^Y \to [0,1]$ depends on at most countably many coordinates, we can take $Y$ to be countable.
\end{proof}

\begin{remark}
    Since the terms of rank $0$ are generated by finitely many terms \cite[Chapter~7]{Abbadini2021}, it follows from the results above that the theory of $\CoAlg(\V)$ is generated by finitely many terms. (We simply add $\Box$ and $\Diamond$ to the theory of $\CompOrd^\op$.)
\end{remark}

\subsection{Upper and lower Vietoris on stably compact spaces}

We now do a similar study for the upper and lower Vietoris functors.

We recall from \cref{n:01up} that $[0,1]^\uparrow$ (resp.\ $[0,1]^\downarrow$) denotes the space whose underlying set is the unit interval $[0,1]$ and whose topology is the set of upsets (resp.\ downsets) of $[0,1]$ that are open in the Euclidean topology.

\begin{notation}
    In this subsection, $\Box$ denotes the function from $\Vu([0,1]^\uparrow)$ to $[0,1]^\uparrow$ that maps $A$ to $\inf A$, while $\Diamond$ denotes the function from $\Vl([0,1]^\downarrow)$ to $[0,1]^\downarrow$ that maps $A$ to $\sup A$.
\end{notation}

\begin{lemma} \label{l:BD-cont-order-p-perfect}
    The functions $\Box \colon \Vu([0,1]^\uparrow) \to [0,1]^\uparrow$ and $\Diamond \colon \Vl([0,1]^\downarrow) \to [0,1]^\downarrow$ are perfect.
\end{lemma}

\begin{proof}
    The proof is similar to the proof of \cref{l:BD-cont-order-p}.
\end{proof}

\begin{proposition} \label{p:it-is-all-the-functions-order-upper}
    Let $\mathcal{C}$ be a set of perfect maps from a stably compact space $X$ to $[0,1]^\uparrow$. Suppose that $\mathcal{C}$ is order-separating and closed under pointwise application of any perfect map from a power of $[0,1]^\uparrow$ to $[0,1]^\uparrow$.
    Then $\mathcal{C}$ is the set of perfect maps from $X$ to $[0,1]^\uparrow$.
\end{proposition}

\begin{proof}
    Up to the isomorphism between $\CompOrd$ and $\StComp$, this is \cref{p:it-is-all-the-functions-order}.
\end{proof}

The following says that the terms of $\CoAlg(\Vu)^\op$ are generated by the terms of $\StComp^\op$ and $\Box$.

\begin{theorem} \label{t:perfect}
    Let $X$ be a set.
    For every perfect map $f \colon \Vu(([0,1]^\uparrow)^X) \to [0,1]^\uparrow$ there are a(n at most countable) set $Y$, a perfect map $s \colon ([0,1]^\uparrow)^Y \to [0,1]^\uparrow$ and a family $(h_y \colon ([0,1]^\uparrow)^X \to [0,1]^\uparrow)_{y \in Y}$ of perfect maps such that
    \[
    f = s \circ \langle\Box \circ \Vu h_y\rangle_{y \in Y}.
    \]
\end{theorem}

\begin{proof}
    Let $\mathcal{C}$ be the set of perfect maps $f \colon \Vu(([0,1]^{\uparrow})^X) \to [0,1]^\uparrow$ for which there are a set $Y$, a perfect map $s \colon ([0,1]^\uparrow)^Y \to [0,1]^\uparrow$ and a family $(h_y \colon ([0,1]^\uparrow)^X \to [0,1]^\uparrow)_{y \in Y}$ of perfect maps such that $f = s \circ \langle\Box \circ \Vu h_y\rangle_{y \in Y}$.
    We shall prove that $\mathcal{C}$ is the set of perfect maps from $\Vu(([0,1]^\uparrow)^X)$ to $[0,1]^\uparrow$.
    As a consequence of \cref{l:BD-cont-order-p-perfect}, every function in $\mathcal{C}$ is perfect.
    Since $\mathcal{C}$ is closed under every perfect map from some power of $[0,1]^\uparrow$ to $[0,1]^\uparrow$, by \cref{p:it-is-all-the-functions-order-upper} it is enough to prove that $\mathcal{C}$ is order-separating, i.e.\ for all distinct $K, L \in \Vu(([0,1]^\uparrow)^X)$ with $K \not\leq L$ (i.e.\ $L \nsubseteq K$) there is a function $f \in \mathcal{C}$ such that $f(K) \not\leq f(L)$.
    Suppose $L \nsubseteq K$.
    Then there is $y \in L \setminus K$.
    By the ordered Urysohn's lemma, there is a perfect map $h \colon ([0,1]^\uparrow)^X \to [0,1]^\uparrow$ such that $h(y) = 0$ and $h[K] \subseteq \{1\}$.
    Then,
    \begin{align*}
        (\Box \circ \Vu h) (K) & = \Box (\Vu h (K)) = \Box (\u h[K]) = \inf \u h[K] = \inf h[K] = 1,\\
        (\Box \circ \Vu h) (L) & = \Box (\Vu h (L)) = \Box (\u h[L]) = \inf \u h[L] = \inf h[L] = 0,
    \end{align*}
    Thus, we have found a function $f = \Box \circ \Vu h \in \mathcal{C}$ such that $f(K) \nleq f(L)$.
    Hence, $\mathcal{C}$ is order-reflecting.
	
    ``At most countability'' follows from the fact that every perfect map $([0,1]^\uparrow)^Y \to [0,1]^\uparrow$ depends on at most countably many coordinates.
\end{proof}

The following says that the terms of $\CoAlg(\Vl)^\op$ are generated by the terms of $\StComp^\op$ and $\Diamond$.

\begin{theorem}
    Let $X$ be a set.
    For every perfect map $f \colon \Vl(([0,1]^{\downarrow})^X) \to [0,1]^\downarrow$ there exist an (at most countable) set $Y$, a perfect map $s \colon ([0,1]^\downarrow)^Y \to [0,1]$ and a family $(h_y \colon ([0,1]^\downarrow)^X \to [0,1]^\downarrow)_{y \in Y}$ of perfect maps such that
    \[
    f = s \circ \langle\Diamond \circ \Vu h_y \rangle_{y \in Y}.
    \]
\end{theorem}

\begin{proof}
    The proof is analogous to the proof of \cref{t:perfect}.
\end{proof}


\subsection*{Acknowledgements}
We are grateful to the organizers of the CT20$\to$21 in Genova for giving us the opportunity to start this collaboration. We also thank the referee for their careful reading and suggestions, which improved the paper.

Marco Abbadini is grateful to Luca Reggio for some important suggestions.
Marco Abbadini's research was supported by the Italian Ministry of University and Research through the PRIN project n.\ 20173WKCM5 \emph{Theory and applications of resource sensitive logics}.

Ivan Di Liberti was supported by the Swedish Research Council (SRC, Vetenskapsrådet) under Grant No. 2019-04545. The research has received funding from the Knut and Alice Wallenberg Foundation through the Foundation’s program for mathematics.

\bibliography{Biblio}
\bibliographystyle{plain}
\end{document}